\documentclass[12pt]{amsart}

\usepackage{amssymb}
\usepackage{esint}
\usepackage{hyperref}
\usepackage{xparse}
\usepackage{float}
\usepackage{tikz}
\usetikzlibrary{arrows}

\usepackage{enumitem}

\usepackage{graphicx}

\makeatletter
\@namedef{subjclassname@2020}{%
  \textup{2020} Mathematics Subject Classification}
\makeatother

\usepackage[T1]{fontenc}

\makeatletter
\newtheorem*{rep@theorem}{\rep@title}
\newcommand{\newreptheorem}[2]{%
\newenvironment{rep#1}[1]{%
 \def\rep@title{#2 \ref{##1}}%
 \begin{rep@theorem}}%
 {\end{rep@theorem}}}
\makeatother

\newtheorem{theorem}{Theorem}[section]
\newreptheorem{theorem}{Theorem}
\newtheorem{corollary}[theorem]{Corollary}
\newtheorem{lemma}[theorem]{Lemma}
\newtheorem{prop}[theorem]{Proposition}

\newtheorem{question}[theorem]{Question}

\theoremstyle{definition}
\newtheorem{definition}[theorem]{Definition}

\numberwithin{equation}{section}

\def \D {\mathcal{D}}
\def \K {\mathcal{K}}
\def \B {\mathcal{B}}
\def \F {\mathcal{F}in}

\def \U {\mathcal{U}}

\def \R {\mathbb{R}}
\def \N {\mathbb{N}}

\def \I {\mathcal{I}}
\def \J {\mathcal{J}}

\def \comp {\circ}
\def \sbs {\subseteq}
\def \sps {\supseteq}
\def \cross {\times}
\def \eps {\varepsilon}
\def \im {\mathrm{im}}
\def \closure {\overline}
\def \one {\mathsf{1}}

\def \ulim {\mathcal{U}\text{-}\lim}

\NewDocumentCommand\Lip{mg}{\ensuremath{\mathrm{Lip}_0\left(#1 \IfNoValueTF{#2}{}{,#2}\right)}}
\NewDocumentCommand\lip{mg}{\ensuremath{\mathrm{lip}_0\left(#1 \IfNoValueTF{#2}{}{,#2}\right)}}
\NewDocumentCommand\LF{mg}{\ensuremath{\mathrm{LF}\left(#1 \IfNoValueTF{#2}{}{,#2}\right)}}
\NewDocumentCommand\LFf{mg}{\ensuremath{\mathrm{LF}_\mathrm{fin}\left(#1 \IfNoValueTF{#2}{}{,#2}\right)}}
\NewDocumentCommand\supp{mg}{\ensuremath{\IfNoValueTF{#2}{\mathrm{supp}\left(#1\right)}{\mathrm{supp}_{#1}\left(#2\right)}}}

\newcommand{\Cl}[1]{\mathcal{C}\ell\left(#1\right)}
\renewcommand{\l}[2]{\ell^1_{#1}\left( #2 \right)}

\newcommand{\rad}[2]{\mathrm{rad}_{#1}\left(#2\right)}
\newcommand{\diam}[1]{\mathrm{diam}\left( #1 \right)}

\newcommand{\ulm}[1]{\mathrm{ulm}\left(#1\right)}
\newcommand{\ulmc}[1]{\overline{\mathrm{ulm}}\left(#1\right)}

\begin{document}


\baselineskip=17pt


\title{Lipschitz Free Spaces over Locally Compact Metric Spaces}

\author{Chris Gartland}
\address{Department of Mathematics, Texas A\&M University, College Station, TX 77843, USA}

\date{}

\begin{abstract}
We prove that the Lipschitz free spaces over certain types of discrete metric spaces have the Radon-Nikod\'ym property. We also show that the Lipschitz free space over a complete, locally compact metric space has the Schur or approximation property whenever the Lipschitz free space over each compact subset also has this property.
\end{abstract}


\maketitle


\tableofcontents

\section{Introduction}
A Banach space $V$ has the \emph{Radon-Nikod\'ym property} (RNP) if every Lipschitz map from $\R \to V$ is differentiable Lebesgue-almost everywhere. We also say that $V$ is an \emph{RNP space}. Important examples of RNP spaces are reflexive spaces and separable duals (\cite[Corollary 2.15]{P}). There is a deep theory of non-biLipschitz embeddability of metric measure spaces into RNP spaces using differentiation methods. These methods were chiefly pioneered by Cheeger-Kleiner, culminating in \cite[Theorem 1.6]{CKRNP}. This theorem implies that nonabelian Carnot groups (\cite[Theorem 6.1]{CKBanach}), inverse limits of certain systems of graphs (\cite[Theorem 10.2]{CKgraphs}), and the Bourdon-Pajot spaces (\cite[Corollary 1.7]{CKRNP}) do not biLipschitz embed into any RNP space.

We make two observations concerning this theory. The first is that differentiation methods are inherently local; they actually rule out \emph{local} biLipschitz embeddability. The second is that, as far as we are aware, differentiation methods (or closely related martingale methods, such as \cite[Theorem 1.3]{O}) are currently the \emph{only} known methods used to prove non-biLipschitz embeddability of metric spaces into RNP spaces. Thus we consider the following question to be a natural one.

\begin{question} \label{q:localembed}
If a complete metric space locally biLipschitz embeds into RNP spaces, must the entire space?
\end{question}

When the metric space is complete and locally compact, this is equivalent to:

\begin{question} \label{q:compactembed}
If every compact subset of a metric space biLipschitz embeds into an RNP space, must the entire space?
\end{question}

An example where the hypothesis is trivially satisfied is when the metric space is discrete. In this case, Question \ref{q:compactembed} takes the form:

\begin{question} \label{q:discreteembed}
Does every complete, discrete metric space biLipschitz embed into an RNP space?
\end{question}

The strongest partial result towards a positive answer to Question \ref{q:discreteembed} is due to Kalton. A metric space $(X,d)$ is \emph{uniformly discrete} if there exists $\theta > 0$ such that $d(x,y) > \theta$ for all $x \neq y \in X$.

\begin{theorem}[\cite{K}, Proposition 4.4] \label{thm:unifdiscrete}
If $X$ is uniformly discrete, then $\LF{X}$ has the RNP.
\end{theorem}

Before proceeding, let us recall the definition of $\LF{X}$. Let $(X,d)$ be a metric space with distinguished basepoint $0 \in X$. Let $\Lip{X}$ denote the Banach space of Lipschitz functions $f: X \to \R$ satisfying $f(0)= 0$ equipped with the norm $\|f\| := \sup_{p \neq q} \frac{|f(p)-f(q)|}{d(p,q)}$. The space $X$ isometrically embeds into $\Lip{X}^*$ via $p \mapsto \delta_p$, where $\delta_p(f) = f(p)$. The linear span of $\{\delta_p\}_{p \in X}$ in $\Lip{X}^*$ is denoted by $\LFf{X}$, and its closure by $\LF{X}$. The space $\LF{X}$ is called the \emph{Lipschitz free space over $X$}. Lipschitz free spaces are a very well-studied class of Banach spaces. See \cite[Chapter 10]{Obook} and \cite{W} (note that Lipschitz free spaces are called Arens-Eells spaces in that text) for textbook introductions to Lipschitz free spaces. Recent popularity in the research of Lipschitz free spaces (and the use of the name ``Lipschitz free space") is due much in part to the articles \cite{GK} and \cite{K}.

Our first result is a generalization of Kalton's theorem and a step closer toward a positive answer to Question \ref{q:discreteembed}.

\begin{definition}
Let $(X,d)$ be a metric space. For $r > 0$ and $p \in X$, $\boldsymbol{B_r(p)} := \{x \in X: d(x,p) \leq r\}$. We define the class of $\boldsymbol{n}$\textbf{-discrete} metric spaces recursively. A metric space is $0$-discrete if it is finite. A metric space $X$ is $(n+1)$-discrete if there exists $\theta > 0$ such that $B_\theta(p)$ is $n$-discrete for all $p \in X$. 
\end{definition}

Here are some quick facts about $n$-discrete metric spaces.
\begin{itemize}
\item Every $n$-discrete metric space is discrete and complete.
\item Uniformly discrete metric spaces are $1$-discrete.
\item The union of finitely many $n$-discrete metric spaces inside an ambient metric space is again $n$-discrete.
\item Every $n$-discrete metric space is also $(n+1)$-discrete, but not conversely. Here is a way to construct $(n+1)$-discrete metric spaces that are not $n$-discrete: Obviously, there are $1$-discrete metric spaces that are not $0$-discrete (for example, any infinite uniformly discrete space). Now take any metric space $(X,d)$ with $\diam{X} = 1$ that is $n$-discrete but not $(n-1)$-discrete. Let $Y$ be the disjoint union of the spaces $\{(X,\frac{1}{k}d)\}_{k=1}^\infty$, where the distance between any two points in different copies of $X$ is 1. Then $Y$ is $(n+1)$-discrete but not $n$-discrete.
\end{itemize}

\begin{reptheorem}{thm:RNP}
For all $n$-discrete metric spaces $X$, $\LF{X}$ has the RNP.
\end{reptheorem}

An entirely separate motive for proving Theorem \ref{thm:RNP} is the intrinsic interest in the structure of Lipschitz free spaces. Indeed, Question \ref{q:compactembed} has a direct counterpart for Lipschitz free spaces.

\begin{question}
If $X$ is complete and locally compact and $\LF{K}$ has the RNP for every $K \sbs X$ compact, must $\LF{X}$ have the RNP?
\end{question}

Although we are unable to answer this question as stated, we provide a positive answer when the RNP is replaced by the Schur property or the approximation property. Recall that a Banach space has the \emph{Schur property} if every weakly null sequence is norm null, and the \emph{approximation property} (AP) if every bounded operator on the space is the limit of a net of finite rank operators with respect to the topology of uniform convergence on compacta (henceforth, ucc). In fact, we are able to prove something more general.

\begin{definition}
Let $X$ be a metric space. Let $\boldsymbol{\Cl{X}}$ denote the set of closed subsets of $X$. A collection $\I \sbs \Cl{X}$ is called an \textbf{ideal} if
\begin{itemize}
\item $\emptyset \in \I$,
\item $\Cl{X} \ni A \sbs B \in \I$ implies $A \in \I$,
\item $A,B \in \I$ implies $A \cup B \in \I$.
\end{itemize}
When $\I \sbs \Cl{X}$ is an ideal, a closed subset $E \sbs X$ is said to be \textbf{locally modeled on} $\boldsymbol{\I}$, or just \textbf{locally-}$\boldsymbol{\I}$, if for every $p \in E$, there exists a radius $\theta > 0$ (depending on $p$) such that $E \cap B_\theta(p) \in \I$.
\end{definition}

\begin{reptheorem}{thm:Schur}
Let $X$ be a complete metric space and $\I \sbs \Cl{X}$ an ideal such that $X$ is locally-$\I$. If $\LF{A}$ has the Schur property for every $A \in \I$, then $\LF{X}$ has the Schur property.
\end{reptheorem}

\begin{reptheorem}{thm:AP}
Let $X$ be a complete metric space and $\I \sbs \Cl{X}$ an ideal such that $X$ is locally-$\I$. If, for every $A \in \I$, there exists $B \in \Cl{X}$ such that $A \sbs B$ and $\LF{B}$ has the AP, then $\LF{X}$ has the AP.
\end{reptheorem}

The main new tool used in the proofs of these theorems is Theorem \ref{thm:ulmclocallyI}, which involves the notion of the \emph{uniformly locally modeled-closure} (ulm-closure) of ideals. This notion bridges the gap between the quantitative notion of ``uniformly locally-$\I$" (see Definition \ref{def:ulm}) and the qualitative notion of locally-$\I$. The bridge is needed because we make only qualitative assumptions on the metric space, but our dealing with Lipschitz free spaces necessitates quantitative arguments at some point.

Using \cite{LancienSchur}, \cite{Daletcompact}, and \cite{Lanciendoubling}, we obtain immediate corollaries of Theorems \ref{thm:Schur} and \ref{thm:AP}.

\begin{corollary}
The Lipschitz free space over every complete, locally countable, locally compact metric space has the Schur property.
\end{corollary}

\begin{proof}
This follows from Theorem \ref{thm:Schur}, \cite[Theorem 3.1]{LancienSchur}, and the fact that the countable, compact subsets of a metric space form an ideal. The theorem \cite[Theorem 3.1]{LancienSchur} states that the Lipschitz free space over every countable compact metric space has the Schur property.
\end{proof}

Recall that a metric space $X$ is \emph{doubling} if there exists $N \in \N$ such that for every $r > 0$ and $p \in X$, there exists $\{p_i\}_{i=1}^N \sbs X$ such that $B_r(p) \sbs \cup_{i=1}^N B_{r/2}(p_i)$.

\begin{corollary}
The Lipschitz free space over every complete metric space that is locally countable and locally compact or locally doubling has the AP (in particular, this holds for every connected complete Riemannian manifold).
\end{corollary}

\begin{proof}
This follows from Theorem \ref{thm:AP}, \cite[Theorem 3.1]{Daletcompact}, \cite[Corollary 2.2]{Lanciendoubling}, and the fact that countable, compact, and doubling subsets of a metric space form ideals. The theorem \cite[Theorem 3.1]{Daletcompact} implies that the Lipschitz free space over a countable, compact metric space has the AP, and \cite[Corollary 2.2]{Lanciendoubling} implies that the Lipschitz free space over a doubling metric space has the AP.
\end{proof}

We recommend \cite{Petitjean} for additional results on the Schur property of Lipschitz free spaces (also see \cite[Section 4]{APP2}), and \cite{LancienSchur} for the AP. We conclude the introduction by recording additional metric spaces whose Lipschitz free space is known to have the RNP.
\begin{itemize}
\item Spaces that biLipschitz embed into a proper metric space distorted by a nontrivial gauge or into a separable dual Banach space distorted by a nontrivial gauge (\cite[Proposition 6.3]{K}).
\item Proper countable spaces (\cite[Theorem 2.1]{Daletproper}).
\item Spaces that biLipschitz embed into an ultrametric space (\cite[Theorem 2]{Cuth}).
\item Separable, complete metric spaces with null Hausdorff-1 measure that biLipschitz embed into $\R$-trees (\cite[Theorem 1.2]{APP1}).
\end{itemize}

\subsection{Related Work}
A completely independent article containing results generalizing Theorem \ref{thm:Schur} was announced right around the same time as ours. Theorem \ref{thm:Schur} implies that $\LF{X}$ has the Schur property whenever $X$ is complete and locally compact and $\LF{K}$ has the Schur property for every $K \sbs X$ compact, and \cite[Corollary 2.6]{APP2} implies the same conclusion assuming only that $X$ is complete and not locally compact.

\section{Preliminaries, Notation, Conventions} \label{sec:prelims}
We cite \cite{W} for references on Lipschitz free spaces. We'll recall three fundamental facts about Lipschitz free spaces that will be used throughout this article, often without reference. Fix a metric space $(X,d)$ with basepoint $0 \in X$. The first fundamental fact is that $\LF{X}^* = \Lip{X}$ (\cite[Theorem 2.2.2]{W}) and that, on bounded sets, the weak* topology is the topology of pointwise convergence. Let $\Delta \sbs X \cross X$ denote the diagonal and set $\tilde{X} := X \cross X \setminus \Delta$. Then $d$ is nonvanishing on $\tilde{X}$. Let $\ell^1(\tilde{X})/d$ denote the Banach space of functions $b: \tilde{X} \to \R$ equipped with the norm $\|b\| := \sum_{(p,q)}|b_{(p,q)}|d(p,q)$. The second fundamental fact is that there is a linear quotient map $\pi: \ell^1(\tilde{X})/d \to \LF{X}$ defined by $\pi(b) = \sum_{(p,q)}b_{(p,q)}(\delta_p-\delta_q)$. The third fundamental fact is that if $0 \in Y \sbs X$, the natural inclusion $\LF{Y} \hookrightarrow \LF{X}$ is an isometric embedding. This is due to the \emph{McShane extension theorem}: every Lipschitz function from $Y$ to $\R$ can be extended to a Lipschitz function on all of $X$ without increasing the Lipschitz norm (\cite[Theorem 1.5.6(a)]{W}).

For $f \in \Lip{X}$, we define $\boldsymbol{\supp{f}} := \closure{f^{-1}(\R \setminus \{0\})}$. For $v \in \LFf{X}$, we define $\boldsymbol{\supp{v}} := \bigcap \{E \sbs X: v \in \text{span}\{\delta_p\}_{p \in E}\}$. An obvious but important observation is that if $f$ and $g$ agree on $\supp{v}$, then $v(f) = v(g)$.

We also need a fundamental result of Kalton.

\begin{theorem}[\cite{K}, Lemma 4.2 and Proposition 4.3] \label{thm:Kalton}
Let $X$ be a metric space and let $P$ be either the Schur, Radon-Nikod\'ym, or approximation property. If $\LF{Y}$ has $P$ for all $Y \sbs X$ closed and bounded, then $\LF{X}$ has $P$.
\end{theorem}

The diameter of $X$ is defined by $\boldsymbol{\diam{X}} := \sup_{x,y \in X} d(x,y)$. Whenever $\diam{X} = 1$, we assume it is equipped with basepoint $0 \in X$ satisfying $d(0,p) = 1$ for all $p \in X$. There is no loss in generality in making this assumption as far as any of the theorems in this article are concerned. Note that this implies, for any $f \in \Lip{X}$, $\|f\|_\infty = \sup_{p \neq 0} \frac{|f(p)-f(0)|}{d(0,p)} \leq \|f\|$, and also, for any $p \in X \setminus \{0\}$, $\|\delta_p\| = 1$.

When constructing Lipschitz functions in $\Lip{X}$, we routinely ignore the requirement that $f$ vanishes at the basepoint. This is because the requirement can easily be met by subtracting a constant from the function, and this will never affect any of the other desired properties of our construction.

\section{ulm-Closure of Ideals in $\Cl{X}$} \label{sec:ideals}
Throughout this section, fix a metric space $(X,d)$. The main result of this section is Theorem \ref{thm:ulmclocallyI}, which is used to prove Theorems \ref{thm:Schur} and \ref{thm:AP}. \\

For $r > 0$ and $E \sbs X$, define $\boldsymbol{B_r(E)} = \cup_{p \in E} B_r(p)$.

\begin{definition} \label{def:ulm}
A collection $\D \sbs \Cl{X}$ of closed subsets is \textbf{downward-closed} if $\emptyset \in \D$ and $\Cl{X} \ni A \sbs B \in \D$ implies $A \in \D$. Let $\D$ be a downward-closed collection. A set $E \in \Cl{D}$ is \textbf{uniformly locally modeled on} $\boldsymbol{\D}$, or just \textbf{uniformly locally-}$\boldsymbol{\D}$, if there exists $\theta > 0$ (dependent on $E$, independent of $F$) such that for all finite $F \sbs X$, $E \cap B_\theta(F) \in \D$. The collection of all such $E \in \Cl{X}$ is denoted \textbf{ulm}$\boldsymbol{(\D)}$. The downward-closed condition implies $\D \sbs \ulm{\D}$ and that $\ulm{\D}$ is also downward-closed. Also, whenever $\I$ is an ideal, $\ulm{\I}$ is an ideal. A downward-closed collection $\D'$ is \textbf{ulm-closed} if $\ulm{\D'} = \D'$. The \textbf{ulm-closure} of a downward-closed collection $\D$, denoted $\boldsymbol{\ulmc{\D}}$, is the intersection of all downward-closed, ulm-closed collections containing $\D$. It is easy to verify that $\ulmc{\D}$ is always itself downward-closed and ulm-closed.
\end{definition}

\begin{prop} \label{prop:ulmcideal}
For all ideals $\I \sbs \Cl{X}$, $\ulmc{\I}$ is an ideal.
\end{prop}

\begin{proof}
Let $\I \sbs \Cl{X}$ be an ideal. The subset $\ulmc{\I}$ is downward-closed by definition, so we only need to show closure under finite unions. Set $\B := \{B \in \Cl{X}: \forall A \in \I, A \cup B \in \ulmc{\I}\}$. Obviously $\I \sbs \B$, and we'll show that $\B$ is ulm-closed. Let $B \in \ulm{\B}$. Let $\theta > 0$ such that $B \cap B_\theta(F) \in \B$ for all $F \sbs X$ finite. Let $A \in \I$. Then for all $F \sbs X$ finite, $(A \cup B) \cap B_\theta(F) = (A \cap B_\theta(F)) \cup (B \cap B_\theta(F)) \in \ulmc{\I}$ by definition of $\B$. Since $\ulmc{\I}$ is ulm-closed, this shows $A \cup B \in \ulmc{\I}$. Since $A \in \I$ was arbitrary, this shows $B \in \B$ and thus $\B = \ulm{\B}$. We can conclude that whenever $A \in \I$ and $B \in \ulmc{\I}$, $A \cup B \in \ulmc{\I}$. Then we can consider the set $\B' := \{B \in \Cl{X}: \forall A \in \ulmc{I}, A \cup B \in \ulmc{\I}\}$, so that the last sentence implies $\I \sbs \B'$. By repeating the same argument, we get that $\B'$ is ulm-closed and thus $\ulmc{\I} \sbs \B'$. By definition of $\B'$ this means $A \cup B \in \ulmc{\I}$ whenever $A,B \in \ulmc{\I}$.
\end{proof}

\begin{theorem} \label{thm:ulmclocallyI}
For every ideal $\I \sbs \Cl{X}$, $\ulmc{\I}$ equals the collection of all locally-$\I$ subsets of $X$.
\end{theorem}

\begin{proof}
Let $\I \sbs \Cl{X}$ be an ideal. Let $\D$ denote the downward-closed collection of all locally-$\I$ closed subsets of $X$. Obviously $\I \sbs \D$ and $\D$ is ulm-closed, so we get $\ulmc{\I} \sbs \D$. To prove the other containment, we prove that $\Cl{X} \setminus \ulmc{\I} \sbs \Cl{X} \setminus \D$. Let $E \in \Cl{X} \setminus \ulmc{\I}$. Choose a sequence of radii $\theta_i \to 0$. Then there is a finite subset $F \sbs X$ such that $E \cap B_{\theta_1}(F) \not\in \ulmc{\I}$. Since $B_{\theta_1}(F) = \cup_{p \in F} B_{\theta_1}(p)$ and $\ulmc{\I}$ is an ideal by Proposition \ref{prop:ulmcideal}, the previous sentence implies $E \cap B_{\theta_1}(p_1) \not\in \ulmc{\I}$ for some $p_1 \in F$. Repeating this argument again yields a point $p_2 \in X$ such that $E \cap B_{\theta_1}(p_1) \cap B_{\theta_2}(p_2) \not\in \ulmc{\I}$. Iterating, we can find a sequence of points $p_i \in X$ such that $\cap_{i=1}^N E \cap B_{\theta_i}(p_i) \not\in \ulmc{\I}$ for all $N \geq 1$. Let $p_i'$ be any sequence of points in $X$ with $p_N' \in \cap_{i=1}^N E \cap B_{\theta_i}(p_i)$. Since $\theta_i \to 0$, $E \in \Cl{X}$, and $X$ is complete, $p'_i \to p$ for some $p \in E$. For any $\theta > 0$, $E \cap B_\theta(p)$ contains $\cap_{i=1}^N E \cap B_{\theta_i}(p_i)$ for some sufficiently large $N$. Since $\ulmc{\I}$ is downward-closed and $\cap_{i=1}^N E \cap B_{\theta_i}(p_i) \not\in \ulmc{\I}$, this implies $E \cap B_\theta(p) \not\in \ulmc{\I}$. In particular, $E \cap B_\theta(p) \not\in \I$. Hence $E \not\in \D$.
\end{proof}

\section{The Schur Property}
The main result of this section is Theorem \ref{thm:Schur}.

\begin{definition} \label{def:T}
Whenever $(X,d)$ is a bounded metric space with basepoint $0 \in X$, $E \sbs X$, and $\theta > 0$, define the bounded operator $T^*_{E,\theta}: \Lip{X} \to \Lip{X}$ by $T^*_{E,\theta}(f)(x) = f(x)\cdot\max\{0,\min\{1,2-2\theta^{-1}d(E,x)\}\}$.
\end{definition}

\begin{prop} \label{prop:T}
Let $X$ be a metric space. Assume $\diam{X} = 1$. For any $E \sbs X$ and $\theta > 0$, $T^*_{E,\theta}$ has a predual $T_{E,\theta}: \LF{X} \to \LF{X}$ and
\begin{itemize}
\item $\|T^*_{E,\theta}\|,\|T_{E,\theta}\|$ are bounded by a function of $\theta$, $C(\theta)$ (independent of $E$).
\item For any $f \in \Lip{X}$ and $v \in \LFf{X}$, $\supp{T^*_{E,\theta}(f)} \sbs \supp{f} \cap B_\theta(E)$ and $\supp{T_{E,\theta}(v)} \sbs \supp{v} \cap B_\theta(E)$.
\item For any $v \in \LFf{X}$, $\supp{v-T_{E,\theta}(v)} \sbs \supp{v} \setminus B_{\theta/2}(E)$.
\end{itemize}
\end{prop}

\begin{proof}
That $T_{E,\theta}^*$ has a predual is due to the fact that it preserves pointwise convergence of functions. The boundedness of $T_{E,\theta}$ comes from the inequality $\|fg\| \leq \|f\|\|g\|_\infty + \|f\|_\infty\|g\| \leq 2\|f\|\|g\|$. The other statements are straightforward to check and we leave the details to the reader.
\end{proof}

\begin{theorem} \label{thm:Schur}
Let $X$ be a complete metric space and $\I \sbs \Cl{X}$ an ideal such that $X$ is locally-$\I$. If $\LF{A}$ has the Schur property for every $A \in \I$, then $\LF{X}$ has the Schur property.
\end{theorem}

\begin{proof}
Assume $\LF{A}$ has the Schur property for all $A \in \I$. By Theorem \ref{thm:Kalton}, we may assume $X$ is bounded. By scaling the metric, we may assume $\diam{X} = 1$. Let $\D$ denote the downward-closed collection of all subsets $E \sbs X$ such that $\LF{E}$ has the Schur property. By assumption, $\I \sbs \D$. We will show $\D$ is ulm-closed, and thus $X \in \D$ by Theorem \ref{thm:ulmclocallyI} and the assumption that $X$ is locally-$\I$.

Suppose $G \in \ulm{\D}$. Let $v_n \in \LF{G}$ such that $v_n$ is weakly null. It suffices to assume $v_n$ belongs to the dense subspace $\LFf{G} \sbs \LF{G}$, and it suffices to find a norm null subsequence. Let $\theta > 0$ such that for all $F \sbs X$ finite, $G \cap B_\theta(F) \in \D$. Set $N_1 := 1$ and $F_1 := \supp{v_1}$. By Proposition \ref{prop:T}, $T_{F_1,\theta}(v_n) \in \LF{G \cap B_\theta(F_1)}$. Thus, by definition of $\D$, $T_{F_1,\theta}(v_n)$ is norm null. Then we can find $N_2 > N_1$ large enough so that $\|T_{F_1,\theta}(v_{N_2})\| \leq 1$. Set $F_2 := F_1 \cup \supp{v_{N_2}}$.  By Proposition \ref{prop:T}, $T_{F_2,\theta}(v_n) \in \LF{G \cap B_\theta(F_2)}$. Thus, by definition of $\D$, $T_{F_2,\theta}(v_n)$ is norm null. Then we can find $N_3 > N_2$ large enough so that $\|T_{F_2,\theta}(v_{N_3})\| \leq 1/2$, and we set $F_3 := F_2 \cup \supp{v_{N_3}}$. Continuing in this way, we can find an increasing sequence of finite subsets $F_1 \sbs F_2 \sbs \dots \sbs X$ and numbers $N_1 < N_2 < \dots$ such that $\supp{v_{N_i}} \sbs F_i$ and $\|T_{F_i,\theta}(v_{N_{i+1}})\| \leq 1/i$ for all $i \geq 1$. Then we form a new sequence $\tilde{v}_i$ defined by $\tilde{v}_1 := v_1$ and $\tilde{v}_{i+1} := v_{N_{i+1}} - T_{F_i,\theta}(v_{N_{i+1}})$ for $i \geq 1$. The sequence $\tilde{v}_i$ is norm-equivalent to a subsequence of $v_n$, so it suffices to prove that this sequence is norm null. By Proposition \ref{prop:T}, $d(\supp{\tilde{v}_{i}},\supp{\tilde{v}_{j}}) \geq \theta/2$ for all $i \neq j$. Let $f_i \in B_{\Lip{\supp{\tilde{v}_i}}}$ such that $\tilde{v}_i(f_i) = \|\tilde{v}_i\|$. Then we define $f(x) := f_i(x)$ if $x \in \supp{\tilde{v}_i}$, and note that the inequality $d(\supp{\tilde{v}_{i}},\supp{\tilde{v}_{j}}) \geq \theta/2$ and the assumption $\diam{X} = 1$ imply $\|f\| \leq \sup_i \{\|f_i\|,\frac{\|f_i\|_\infty+\|f_i\|_\infty}{\theta/2}\} \leq 4\theta^{-1}$. Thus, since $\tilde{v}_i$ is weakly null, we have $\|\tilde{v}_i\| = \tilde{v}_i(f) \to 0$. This proves that $\LF{G}$ has the Schur property. Hence, $G \in \D$ and $\ulm{\D} \sbs \D$.
\end{proof}

\section{The Approximation Property}
The main result of this section is Theorem \ref{thm:AP}.

Recall the definition of $T_{F,\theta}$ from Definition \ref{def:T}. Whenever $X$ is a metric space, the finite subsets $F$ of $X$ form a directed system under the inclusion relation, and thus we get a uniformly bounded net of operators $(T_{F,\theta})_{F \sbs X}$ in $\B(\LF{X},\LF{X})$ for every $\theta > 0$.

\begin{prop} \label{prop:Tconv}
Let $X$ be a metric space. Assume $\diam{X} = 1$. For every $\theta > 0$, $T_{F,\theta}$ ucc-converges to Id$_{\LF{X}}$ as $F \to \infty$, $F$ finite.
\end{prop}

\begin{proof}
Since $\|T_{F,\theta}\|$ is uniformly bounded by a function of $\theta$ (independent of $F$), it suffices to check that $T_{F,\theta}$ pointwise converges to Id$_{\LF{X}}$ as $F \to \infty$ on a dense subset of $\LF{X}$. This happens on the dense subset $\LFf{X} \sbs \LF{X}$ since whenever $v \in \LFf{X}$ and $F \sps \supp{v}$, $T_{F,\theta}(v) = v$.
\end{proof}

\begin{theorem} \label{thm:AP}
Let $X$ be a complete metric space and $\I \sbs \Cl{X}$ an ideal such that $X$ is locally-$\I$. If, for every $A \in \I$, there exists $B \in \Cl{X}$ such that $A \sbs B$ and $\LF{B}$ has the AP, then $\LF{X}$ has the AP.
\end{theorem}

\begin{proof}
Assume that, for every $A \in \I$, there exists $B \in \Cl{X}$ such that $A \sbs B$ and $\LF{B}$ has the AP. By Theorem \ref{thm:Kalton}, we may assume $X$ is bounded. By scaling the metric, we may assume $\diam{X} = 1$. Let $\D$ denote the downward-closed collection of all sets $E \sbs X$ such that whenever $T \in \B(\LF{X},\LF{X})$ and $\im(T) \sbs \LF{E}$, then $T$ belongs to the ucc-closure of the finite rank operators on $\LF{X}$. By assumption, $\I \sbs \D$. We will show $\D$ is ulm-closed, and thus $X \in \D$ by Theorem \ref{thm:ulmclocallyI} and the assumption that $X$ is locally-$\I$.

Suppose $G \in \ulm{\D}$. Let $T \in \B(\LF{X},\LF{X})$ such that $\im(T) \sbs \LF{G}$. Let $\theta > 0$ such that for all $F \sbs X$ finite, $G \cap B_\theta(F) \in \D$. By Proposition \ref{prop:T}, for each $F \sbs X$ finite, $\im(T_{F,\theta} \comp T) \sbs \LF{G \cap B_\theta(F)}$. By definition of $\D$, this implies each $T_{F,\theta} \comp T$ belongs to the ucc-closure of the finite rank operators on $\LF{X}$. By Proposition \ref{prop:Tconv}, the net of operators $T_{F,\theta} \comp T$ ucc-converges to $T$ as $F \to \infty$. This shows $T$ belongs to the ucc-closure of the finite rank operators on $\LF{X}$, and thus $G \in \D$ and $\ulm{\D} \sbs \D$.
\end{proof}

\section{The Radon-Nikod\'ym Property} \label{sec:RNP}
The goal of this section is to prove Theorem \ref{thm:RNP}, which occurs in the next subsection. The proofs of the supporting lemmas follow. First, we define the key players.

\begin{definition}
Let $K \in [1,\infty)$. A bounded linear map $T: V \to W$ between Banach spaces is a $\boldsymbol{K}$\textbf{-semi-embedding} if it is injective and $\closure{T(B_V)} \sbs T(KB_V)$, where $B_V$ denotes the closed unit ball of $V$.
\end{definition}

The key fact we need about $K$-semi-embeddings is that they preserve the RNP for separable spaces.

\begin{theorem}[\cite{P}, Proposition 2.42] \label{thm:semiembeddingRNP}
A separable Banach space that $K$-semi-embeds into an RNP space has the RNP.
\end{theorem}

The theorem is actually stated and proved only for $K=1$ in \cite{P}, but easily adapts to general $K$. For the remainder of the article, \textbf{fix a metric space} $\boldsymbol{(X,d)}$ with $\boldsymbol{\diam{X} = 1}$ and \textbf{basepoint} $\boldsymbol{0 \in X}$ with $\boldsymbol{d(0,p) = 1}$ for all $\boldsymbol{p \in X}$.

\begin{definition}
Let $\I \sbs \Cl{X}$ be an ideal. Let $\boldsymbol{\Lip{\I,X}}$ denote the (not necessarily closed) subspace of $\Lip{X}$ consisting of functions $f$ with $\supp{f} \in \I$. We stress that the functions in $\Lip{\I,X}$ are required to be globally defined on $X$ and not just on some set in $\I$. The subset $\Lip{\I,X}$ is a subspace since $\I$ is an ideal. For each $p \in X$, we get a continuous linear functional $\delta_p \in \Lip{\I,X}^*$ defined by $\delta_p(f) = f(p)$. For each subset $Y \sbs X$, define $\boldsymbol{\LFf{\I}{Y}}$ to be the linear span of $\{\delta_p\}_{p \in Y}$ in $\Lip{\I,X}^*$, and $\boldsymbol{\LF{\I}{Y}}$ its closure. We also stress that, despite the absence of `$X$' in the notation, the structure of the space $\LF{\I}{Y} \sbs \Lip{\I,X}^*$ depends not only on $Y$ but also $X$ because the functions in $\Lip{\I,X}$ are required to be globally defined on $X$.

If $v \in \LFf{\I}{X}$, we denote by $\supp{v}$ the smallest finite set $F$ such that $v \in$ span$\{\delta_p\}_{p \in F}$. We denote the norm on $\LF{\I}{X}$ by $\|\cdot\|_\I$. It of course holds that $\Lip{\Cl{X}}{X} = \Lip{X}$, $\LFf{\Cl{X}}{X} = \LFf{X}$, and $\LF{\Cl{X}}{X} = \LF{X}$. Whenever $\I_1 \sps \I_2$, we get a canonical restriction map $\boldsymbol{R^{\I_1}_{\I_2}}: \LFf{\I_1}{X} \to \LFf{\I_2}{X}$ which extends to a linear contraction $R^{\I_1}_{\I_2}: \LF{\I_1}{X} \to \LF{\I_2}{X}$. If $\I_1 = \Cl{X}$, we just write $R_{\I_2}$.
\end{definition}

\subsection{Proof of Theorem \ref{thm:RNP}}
\begin{definition}
Define ideals $\boldsymbol{\F_n} \sbs \Cl{X}$ recursively, as follows:
\begin{itemize}
\item $\F_0 := \boldsymbol{\F} :=$ finite subsets of $X$.
\item $\F_{n+1} := \ulm{\F_n}$.
\end{itemize}
\end{definition}

It holds that $Y \in \F_n$ if and only if $Y$ is $n$-discrete.

\begin{lemma} \label{lem:RNP}
Assume $X$ is countable, $\diam{X} = 1$, and $X$ is $n$-discrete for some $n \in \N$. The space $\LF{\F_k}{X}$ has the RNP for every $k \in \N$.
\end{lemma}

\begin{proof}
We will prove the theorem by induction. The base case is $k=0$, $\LF{\F_0}{X} = \LF{\F}{X}$. Corollary \ref{cor:predual} implies $\LF{\F}{X}$ has the RNP since it is a separable dual. Now assume $\LF{\F_k}{X}$ has the RNP for some $k \in \N$. Then since $\F_{k+1} = \ulm{\F_k}$, Theorem \ref{thm:semiembedding} and the inductive hypothesis imply $R^{\F_{k+1}}_{\F_k}: \LF{\F_{k+1}}{X} \to \LF{\F_k}{X}$ is a 7-semi-embedding from a separable space into an RNP space, hence $\LF{\F_{k+1}}{X}$ has the RNP by Theorem \ref{thm:semiembeddingRNP}.
\end{proof}

\begin{theorem} \label{thm:RNP}
For all $n$-discrete metric spaces $X$, $\LF{X}$ has the RNP.
\end{theorem}

\begin{proof}
Let $X$ be an $n$-discrete metric space. Theorem \ref{thm:Kalton} and the fact that $n$-discreteness passes to arbitrary subsets allows us to assume $X$ is bounded. By scaling, we may assume $\diam{X} = 1$. The RNP is separably determined (\cite[Corollary 2.12]{P}). Then since every separable subspace of $\LF{X}$ is contained in $\LF{Y}$ for some $Y \sbs X$ countable, we may assume $X$ is countable. Then the conclusion follows from Lemma \ref{lem:RNP} and the fact that $\LF{X} = \LF{\Cl{X}}{X} = \LF{\F_n}{X}$.
\end{proof}

The important inputs into these proofs are Corollary \ref{cor:predual} and Theorem \ref{thm:semiembedding}. Corollary \ref{cor:predual} is proved in Section \ref{ss:predual} and follows quickly from standard results in \cite{W}; we will not say any more about the proof. Theorem \ref{thm:semiembedding} states that $R^{\F_{n+1}}_{\F_n}$ is a 7-semi-embedding. Injectivity is straightforward and is proved in Theorem \ref{thm:injective}. The difficult part of the proof is showing that $\closure{R^{\F_{n+1}}_{\F_n}(B_{\LF{\F_{n+1}}{X}})} \sbs R^{\F_{n+1}}_{\F_n}(7B_{\LF{\F_{n+1},X}})$. \\

\subsection{Special Case of Theorem \ref{thm:semiembedding}} ~\\
\indent To gain intuition for why $\closure{R^{\F_{n+1}}_{\F_n}(B_{\LF{\F_{n+1}}{X}})} \sbs R^{\F_{n+1}}_{\F_n}(7B_{\LF{\F_{n+1},X}})$ could be true, we will prove a very special case in the following proposition:

\begin{prop} \label{prop:semiembeddingspecial}
For every ideal $\I \sbs \Cl{X}$ and $p,q_k \in X$, if \\ $R^{\ulm{\I}}_{\I}(\delta_p-\delta_{q_k}) \overset{k \to \infty}{\to} \delta_p$, then $\|\delta_p\|_{\ulm{\I}} \leq \liminf_{k \to \infty} \|\delta_p-\delta_{q_k}\|_{\ulm{\I}}$.
\end{prop}

Before proving Proposition \ref{prop:semiembeddingspecial}, we introduce an important quantity and prove two propositions that clarify its role in the study of $\LF{\I}{X}$.

\begin{definition}
Let $\I \sbs \Cl{X}$ be an ideal. For $p \in X$, we define $\boldsymbol{\rad{\I}{p}} := \sup \{r \leq \diam{X}: B_r(p) \in \I\}$, with the convention that $\sup \emptyset = 0$.
\end{definition}

\begin{prop} \label{prop:radInorm}
For all ideals $\I \sbs \Cl{X}$ and $p,p' \in X$, $\|\delta_p\|_\I = \rad{\I}{p}$ and $\|\delta_p-\delta_{p'}\|_\I = \min\{d(p,p'),\rad{\I}{p}+\rad{\I}{p'}\}$.
\end{prop}

\begin{proof}
Let $\I,p,p'$ be as above. Let $\eps > 0$ and $f \in B_{\Lip{\I}{X}}$. Then $\supp{f} \in \I$, so by definition of $\rad{\I}{p}$, there exists $q \in B_{\rad{\I}{p}+\eps}(p)$ such that $q \notin \supp{f}$. Hence, $d(p,q) \leq \rad{\I}{p}+\eps$ and $f(q) = 0$. Since $\|f\| \leq 1$, it follows that $|\delta_p(f)| = |f(p)| \leq d(p,q) \leq \rad{\I}{p}+\eps$. Since $\eps > 0$ and $f \in B_{\Lip{\I}{X}}$ were arbitrary, it follows that $\|\delta_p\|_\I \leq \rad{\I}{p}$. For the other inequality, it obviously suffices to assume $\rad{\I}{p} > 0$. Let $\eps \in (0,\rad{\I}{p})$. Then define $f: X \to \R$ by $f(x) = \max(0,\rad{\I}{p}-\eps-d(p,x))$. Then $\|f\| \leq 1$ and $\supp{f} \sbs B_{\rad{\I}{p}-\eps}(p) \in \I$. Thus, $\|\delta_p\|_\I \geq |\delta_p(f)| = f(p) = \rad{\I}{p}-\eps$. Since $\eps \in (0,\rad{\I}{p})$ was arbitrary, we get $\|\delta_p\|_\I \geq \rad{\I}{p}$. This proves the first statement.

For the second statement, we clearly have \\ $\|\delta_p-\delta_{p'}\|_\I \leq \min\{d(p,p'),\rad{\I}{p}+\rad{\I}{p'}\}$. For the other inequality, there are three cases:
\begin{enumerate}
\item $d(p,p') \leq \max\{\rad{\I}{p},\rad{\I}{p'}\}$,
\item $\max\{\rad{\I}{p},\rad{\I}{p'}\} \leq d(p,p') \leq \rad{\I}{p}+\rad{\I}{p'}$,
\item $d(p,p') \geq \rad{\I}{p}+\rad{\I}{p'}$.
\end{enumerate}
In the first case, without loss of generality, assume $d(p,p') \leq \rad{\I}{p}$ and use the function $f(x) = \max(0,\rad{\I}{p}-\eps-d(p,x))$. In the second case, define $f(x) = \max(0,\rad{\I}{p}-\eps-d(p,x))$ and $g(x) = \max(0,d(p,p')-\rad{\I}{p}-\eps-d(p',x))$ and use the function $f-g$. In the third case, $f(x) = \max(0,\rad{\I}{p}-\eps-d(p,x))$ and $g(x) = \max(0,\rad{\I}{p'}-\eps-d(p',x))$ and use the function $f-g$.
\end{proof}

\begin{prop} \label{prop:radIinf}
For all ideals $\I \sbs \Cl{X}$, $p \in X$, and \\ $r \in (0,\rad{\ulm{\I}}{p})$, $\inf_{q \in B_r(p)} \rad{\I}{q} > 0$.
\end{prop}

\begin{proof}
Let $\I,p,r$ be as above. Choose $r' \in (r,\rad{\ulm{\I}}{p})$. By definition of $\rad{\ulm{\I}}{p}$, $B_{r'}(p) \in \ulm{\I}$. By definition of $\ulm{\I}$, there exists $\theta' > 0$ such that $B_{r'}(p) \cap B_{\theta'}(q) \in \I$ for every $q \in X$. Set $\theta := \min(\theta',r'-r)$. Then $B_\theta(q) \sbs B_{r'}(p) \cap B_{\theta'}(q) \in \I$ for every $q \in B_r(p)$. Hence, by definition of $\rad{\I}{q}$, $\inf_{q \in B_r(p)} \rad{\I}{q} \geq \theta > 0$.
\end{proof}

\begin{proof}[Proof of Proposition \ref{prop:semiembeddingspecial}]
Let $\I \sbs \Cl{X}$ be an ideal, $p,q_k \in X$, and assume $R^{\ulm{\I}}_{\I}(\delta_p-\delta_{q_k}) \overset{k \to \infty}{\to} \delta_p$. Then $\|\delta_{q_k}\|_\I \overset{k \to \infty}{\to} 0$. By Proposition \ref{prop:radInorm}, this is equivalent to $\rad{\I}{q_k} \overset{k \to \infty}{\to} 0$. Together with Proposition \ref{prop:radIinf}, this implies $\liminf_{k \to \infty} d(p,q_k) \geq \rad{\ulm{\I}}{p}$. Thus, using Proposition \ref{prop:radInorm} again,
$$\|\delta_p\|_{\ulm{\I}} = \rad{\ulm{\I}}{p} \leq \liminf_{k \to \infty} \min\{d(p,q_k),\rad{\ulm{\I}}{p}+\rad{\ulm{\I}}{q_k}\}$$
$$= \liminf_{k \to \infty}\|\delta_p-\delta_{q_k}\|_{\ulm{\I}}.$$
\end{proof}

\subsection{Proof Sketch of Theorem \ref{thm:semiembedding} and Outline of Lemma Structure} ~\\
\indent The proof of Theorem \ref{thm:semiembedding} in full is more complicated and requires a series of supporting lemmas and theorems. Figure \ref{fig:lemmas} shows the dependency structure. Theorem \ref{thm:quotient} is particularly important in understanding the structure of $\LF{\I}{Y}$. It generalizes Proposition \ref{prop:radInorm} and is used in different ways to understand both the domain of $R^{\F_{n+1}}_{\F_n}$ (through Lemma \ref{lem:quotientclose}) and its codomain (through Lemma \ref{lem:quotientfinite}). Consider a sequence $v_k \in B_{\LF{\F_{n+1}}{X}}$ with $R^{\F_{n+1}}_{\F_n}(v_k) \overset{k \to \infty}{\to} u \in B_{\LF{\F_{n}}{X}}$, as would be needed in the proof of Theorem \ref{thm:semiembedding}. Set $r_p := (1-\eps)\rad{\F_{n+1}}{p}$. After throwing away a term from the image of $R_{\F_n}$, Lemma \ref{lem:quotientfinite} lets us assume $u$ is finitely supported in each ball $B_{r_p}(p)$. After throwing away terms from the image of $R_{\F_{n+1}}$, Lemma \ref{lem:quotientclose} allows us to assume $v_k = \tilde{a}_k + \sum_{i=1}^\infty v_k^i$, where $\tilde{a}_k$ comes from a dual space, $\sum_{i=1}^\infty \|v_k^i\|_{\F_{n+1}} \leq 1$, and each $v_k^i$ is supported inside some ball $B_{r_{p_i}}(p_i)$. Lemma \ref{lem:sequence2} is an important technical result that then lets us (eventually in $k$, depending on $i$) reduce the support of $v_k^i$ by intersecting it with the support of $u$. An essential input into the proof of this lemma is Lemma \ref{lem:isomorphism} (via Lemma \ref{lem:sequence1}), which allows us to ``upgrade" our mode of convergence from $\|\cdot\|_{\F_n}$ to $\|\cdot\|_{\F_{n+1}}$ inside the balls $B_{r_p}(p)$. Since each $v_k^i$ is supported in $B_{r_{p_i}}(p_i)$ and $u$ is finitely supported in $B_{r_{p_i}}(p_i)$, this means $v_k^i$ (eventually) belongs to a finite-dimensional space. Thus $v_k = \tilde{a}_k + \sum_{i=1}^\infty v_k^i$, and all the terms come from a dual space. By weak*-compactness, we get an ultralimit $v \in \LF{\F_{n+1}}{X}$ that satisfies $R^{\F_{n+1}}_{\F_n}(v) = u$, as required.

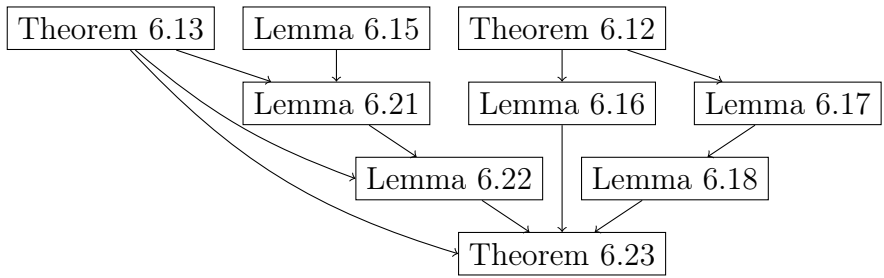
\begin{figure}
\begin{tikzpicture}
\node[rectangle,draw] (injective)  at (-3, 0) {Theorem \ref{thm:injective}};

\node[rectangle,draw] (iso)  at (0, 0) {Lemma \ref{lem:isomorphism}};
\node[rectangle,draw] (quotient) at (3, 0) {Theorem \ref{thm:quotient}};
\node[rectangle,draw] (sequence1) at (0,-1) {Lemma \ref{lem:sequence1}};
\node[rectangle,draw] (quotientclose) at (3,-1) {Lemma \ref{lem:quotientclose}};
\node[rectangle,draw] (quotientseparated) at (6,-1) {Lemma \ref{lem:quotientseparated}};
\node[rectangle,draw] (sequence2) at (1.5,-2) {Lemma \ref{lem:sequence2}};
\node[rectangle,draw] (quotientfinite) at (4.5,-2) {Lemma \ref{lem:quotientfinite}};
\node[rectangle,draw] (semiembedding) at (3,-3) {Theorem \ref{thm:semiembedding}};
\draw[->] (injective) to (sequence1);
\draw[->] (iso) to (sequence1);
\draw[->] (quotient) to (quotientclose);
\draw[->] (quotient) to (quotientseparated);
\draw[->] (injective) to [bend right=10] (sequence2.west);
\draw[->] (sequence1) to (sequence2);
\draw[->] (quotientseparated) to (quotientfinite);
\draw[->] (injective) to [bend right=15] (semiembedding.west);
\draw[->] (sequence2) to (semiembedding);
\draw[->] (quotientclose) to (semiembedding);
\draw[->] (quotientfinite) to (semiembedding);
\end{tikzpicture}
\caption{Dependency structure of the lemmas and theorems in Section \ref{sec:RNP}. There is an arrow from Lemma/Theorem A to Lemma/Theorem B if the proof or statement of Lemma/Theorem B references Lemma/Theorem A.}
\label{fig:lemmas}
\end{figure}

The remaining lemmas and theorems in this section are organized into subsections according to the hypotheses needed on the space $X$ and the ideal $\I$. The subsections are named according to these hypotheses.

\subsection{General $X$, General $\I$}
\begin{definition}
Let $\I \sbs \Cl{X}$ be an ideal. We define $\boldsymbol{\l{\I}{X}}$ to be the Banach space of equivalence classes of real-valued functions $a: X \to \R$ with finite norm $\boldsymbol{\|a\|_{\l{\I}{X}}} := \sum_{p \in X} |a_p|\rad{\I}{p}$. The equivalence relation is $a \sim b$ if $a_p = b_p$ for all $p \in X$ with $\rad{\I}{p} > 0$. We will abuse notation and continue to write $a$ for its equivalence class, but there should never be any trouble verifying well-definedness whenever it arises. Whenever $Y \sbs X$, there is a closed subspace $\l{\I}{Y} = \{a \in \l{\I}{X}: \forall p \in X \setminus Y, a_p = 0\}$. Although `$X$' is absent from the notation, the structure of the space $\l{\I}{Y}$ depends not only on $Y$ but also $X$ because $\rad{\I}{p}$ depends not only on $Y$ but also $X$, even if $p \in Y$. For $a \in \l{\I}{X}$, we define $\supp{a} := \{p \in X: \rad{\I}{p} > 0, a_p \neq 0\}$.

By Proposition \ref{prop:radInorm}, the map $\boldsymbol{Q}_\I: \l{\I}{X} \to \LF{\I}{X}$ defined by $Q_\I(a) = \sum_{p \in X} a_p\delta_p$ is a well-defined linear contraction.
\end{definition}

\begin{theorem} \label{thm:quotient}
For all ideals $\I \sbs \Cl{X}$ and subsets $Y \sbs X$, $R_{\I} \oplus Q_{\I}: \LF{Y} \oplus_1 \l{\I}{Y} \to \LF{\I,Y}$ is a quotient map; it maps the open ball onto the open ball.
\end{theorem}

\begin{proof}
Let $\I,Y$ be as above. Let $v \in \LFf{\I}{Y}$ with finite support $\supp{v} =: F \sbs Y$ and $\|v\|_\I < 1$. We may assume $\rad{\I}{p} > 0$ for all $p \in F$ since otherwise $\delta_p = 0 \in \LF{\I}{Y}$.

Let $\ell^\infty_\I(F)$ denote the dual space (and predual) of $\l{\I}{F}$. Consider the map $T = (T_1,T_2): \Lip{\I}{X} \to \Lip{F} \oplus_\infty \ell^\infty_\I(F)$ defined by $T_1(f) = f \big|_F$ and $T_2(f) = \sum_{p \in F} f(p)\one_p$. So both $T_1$ and $T_2$ are restriction maps of sorts. Set $V := \im(T)$. We claim that $T$ is a quotient onto $V$. Linearity is clear, and contractivity is as well as soon one recalls (from the proof of Proposition \ref{prop:radInorm}) that $|f(p)| \leq \|f\|\rad{\I}{p}$ whenever $f \in \Lip{\I}{X}$. Now let $(f,\sum_{p \in F} f(p)\one_p) \in V$ with $\max_{p \in F} \{\|f\|,\frac{|f(p)|}{\rad{\I}{p}}\} = 1-\eps$ for some $\eps \in (0,1)$. In particular, $|f(p)| \leq (1-\eps)\rad{\I}{p}$ for all $p \in F$. Extend the domain of $f$ to all of $X$ without increasing the Lipschitz norm using the McShane extension theorem, and call this extension $f'$. Define $F_+ := \{p \in F: f(p) > 0\}$ and $F_- := \{p \in F: f(p) < 0\}$. Define $g_+,g_-: X \to \R$ by $g_+(x) = \max_{p \in F_+} \{0,(1-\eps)\rad{\I}{p}-d(p,x)\}$ and $g_-(x) = -\max_{p \in F_-} \{0,(1-\eps)\rad{\I}{p}-d(p,x)\}$. Then $\|g_+\|,\|g_-\| \leq 1$, $g_+ \geq 0$, $g_- \leq 0$, and $\supp{g_+} \sbs \cup_{p \in F_+} B_{(1-\eps)\rad{\I}{p}}(p)$  and $\supp{g_-} \sbs \cup_{p \in F_-} B_{(1-\eps)\rad{\I}{p}}(p)$. Then $f'' := \max(g_-,\min(g_+,f'))$ satisfies $\|f''\| \leq 1$ and $\supp{f} \sbs \cup_{p \in F} B_{(1-\eps)\rad{\I}{p}}(p)$ which belongs to $\I$ since it is a finite union of sets which themselves belong to $\I$ by definition of $\rad{\I}{p}$. Also, the condition $|f(p)| \leq (1-\eps)\rad{\I}{p}$ for all $p \in F$ implies $f''(p) = f(p)$ for all $p \in F$. This implies $T(f'') = (f,\sum_{p \in F} f(p)\one_p)$, proving that $T$ is a quotient map.

$T^*: V^* \to \Lip{\I}{X}^*$ is then an isometric embedding. The image of $T^*$ is exactly the subspace of $\LFf{\I}{X} \sbs \Lip{\I}{X}^*$ consisting of elements supported in $F$. In particular, $v \in \im(T^*)$. Since $V \sbs \Lip{F} \oplus_\infty \ell^\infty_\I(F)$, there is a quotient map $S: \LF{F} \oplus_1 \l{\I}{F} \twoheadrightarrow V^*$ so that $T^* \comp S: \LF{F} \oplus_1 \l{\I}{F} \to \im(T^*)$ is a quotient map. Then since $v \in \im(T^*)$ with $\|v\|_\I < 1$, there are $w \in \LF{F} \sbs \LF{Y}$ and $a \in \l{\I}{F} \sbs \l{\I}{Y}$ such that $T^*(S(w,a)) = v$ and $\|w\| + \|a\|_{\l{\I}{F}} < 1$. It is easily seen that $T^*(S(w,a)) = R_{\I}(w) + Q_{\I}(a)$, proving the theorem.
\end{proof}

\begin{theorem} \label{thm:injective}
For every ideal $\I \sbs \Cl{X}$, $R^{\ulm{\I}}_{\I}: \LF{\ulm{\I}}{X} \to \LF{\I}{X}$ is injective.
\end{theorem}

\begin{proof}
Let $\I \sbs \Cl{X}$ be an ideal. Let $v \in \ker(R^{\ulm{\I}}_{\I})$. Let \\ $f \in B_{\Lip{\ulm{\I}}{X}}$ and $\eps > 0$ be arbitrary. Set $A := \supp{f} \in \ulm{\I}$. Let $\theta > 0$ such that $A \cap B_\theta(F) \in \I$ for every $F \sbs X$ finite. Choose $w \in \LFf{\ulm{\I}}{X}$ such that $F := \supp{w}$ is finite and $\|v-w\|_{\ulm{\I}} < \eps C(\theta)^{-1}$. Set $\tilde{f} := T^*_{F,\theta}(f)$. Then Proposition \ref{prop:T} implies $w(\tilde{f}) = w(f)$ and $\|\tilde{f}\| \leq C(\theta)$. Furthermore, $\supp{\tilde{f}} = A \cap B_\theta(F) \in \I$, so $\tilde{f} \in \Lip{\I}{X}$. Then since $R^{\ulm{\I}}_{\I}(v) = 0$, $v(\tilde{f}) = 0$. Putting these together we get
$$|v(f)| \leq \eps C(\theta)^{-1}\|f\| + |w(f)| = \eps C(\theta)^{-1}\|f\| + |w(\tilde{f})| \leq \eps C(\theta)^{-1} + |w(\tilde{f})|$$
$$\leq \eps C(\theta)^{-1} + \eps C(\theta)^{-1}\|\tilde{f}\| + |v(\tilde{f})| = \eps C(\theta)^{-1} + \eps C(\theta)^{-1}\|\tilde{f}\|$$
$$\leq \eps C(\theta)^{-1} + \eps \leq 2\eps.$$
Since $f \in B_{\Lip{\ulm{\I}}{X}}$ and $\eps > 0$ were arbitrary, this shows $v = 0$.
\end{proof}

\begin{definition}
For a closed subset $A \sbs X$ and ideal $\I \sbs \Cl{X}$, we let $\boldsymbol{\langle A,\I \rangle} \sbs \Cl{X}$ denote the ideal generated by $A$ and $\I$. Precisely, $\langle A,\I \rangle := \{B \cup I: \Cl{X} \ni B \sbs A, I \in \I\}$. We define $\boldsymbol{\langle A \rangle}$ to be the ideal generated by $A$, meaning the collection of closed subsets of $A$.
\end{definition}

\begin{lemma} \label{lem:isomorphism}
For every ideal $\I$ and $A \in \ulm{\I}$, $R^{\langle A,\I \rangle}_\I: \LF{\langle A,\I \rangle}{X} \to \LF{\I}{X}$ is an isomorphism.
\end{lemma}

\begin{proof}
Let $\I \sbs \Cl{X}$ be an ideal and $A \in \ulm{\I}$. Clearly $\im(R^{\langle A,\I \rangle}_\I)$ contains the dense subset $\LFf{\I}{X} \sbs \LF{\I}{X}$. Let $\theta > 0$ such that $A \cap B_\theta(F) \in \I$ for every $F \sbs X$ finite. Let $v \in \LFf{\langle A,\I \rangle}{X}$ with finite support $F := \supp{v}$. Let $f \in B_{\Lip{\langle A,\I \rangle}{X}}$ such that $|v(f)| \geq \|v\|_{\langle A,\I \rangle}/2$. Let $I \in \I$ such that $\supp{f} \sbs A \cup I$. Then by Proposition \ref{prop:T}, $\supp{T^*_{F,\theta}(f)} \sbs \supp{f} \cap B_\theta(F) \sbs (A \cap B_\theta(F)) \cup I \in \I$, and thus $T^*_{F,\theta}(f) \in \Lip{\I}{X}$. Also by Proposition \ref{prop:T}, $\|T^*_{F,\theta}(f)\|$ is bounded by some function of $\theta$, $C(\theta)$ (independent of $f$). Note $T^*_{F,\theta}(f)$ agrees with $f$ on $F$ and so $R^{\langle A,\I \rangle}_\I(v)(T^*_{F,\theta}(f)) = v(f)$. Putting this all together we get
$$\|R^{\langle A,\I \rangle}_\I(v)\|_\I \geq C(\theta)^{-1}|R^{\langle A,\I \rangle}_\I(v)(T^*_{F,\theta}(f))| = C(\theta)^{-1}|v(f)| \geq \frac{C(\theta)^{-1}}{2}\|v\|_{\langle A,\I \rangle}.$$
\end{proof}

\begin{lemma} \label{lem:quotientclose}
For any ideal $\I \sbs \Cl{X}$, $Y \sbs X$, $c \in (0,1)$, and $u \in \LF{\I}{Y}$ with $\|u\|_\I < 1$ there exist, setting $\rho(p,q) := c \min\{\rad{\I}{p},\rad{\I}{q}\}$, $\{b_{(p,q)}\}_{(p,q) \in Y^2: d(p,q) \leq \rho(p,q)}$, and $\{\tilde{a}_p\}_{p \in Y} \sbs \R$ such that
$$u = \sum_{\substack{(p,q) \in Y^2 \\ d(p,q) \leq \rho(p,q)}}b_{(p,q)}(\delta_p-\delta_q) + \sum_{p \in Y} \tilde{a}_{p} \delta_p$$
and
$$\sum_{\substack{(p,q) \in Y^2 \\ d(p,q) \leq \rho(p,q)}}|b_{(p,q)}|d(p,q) + \sum_{p \in Y} |\tilde{a}_{p}|\rad{\I}{p} < 3c^{-1}.$$
\end{lemma}

\begin{proof}
Let $\I$, $Y$, $c$, $u$, be as above.

First we show that whenever $p,q \in X$ such that $d(p,q) > \rho(p,q) = c\cdot \min\{\rad{\I}{p},\rad{\I}{q}\}$, it holds that
\begin{equation} \label{eq:3c-1}
\rad{\I}{p}+\rad{\I}{q} < 3c^{-1}d(p,q).
\end{equation}
Let $p,q$ be such points, and without loss of generality assume $\rad{\I}{q} \leq \rad{\I}{p}$. So then we have $d(p,q) > c \cdot \rad{\I}{q}$ and thus $2c^{-1}d(p,q)>2\rad{\I}{q}$. It is also an immediate consequence of the triangle inequality and the definition of rad${_\I}$ that $d(p,q) \geq \rad{\I}{p}-\rad{\I}{q}$. Adding these two inequalities, and using the fact that $c \leq 1$, gives us $\rad{\I}{p}+\rad{\I}{q} < 3c^{-1}d(p,q)$.

Now, by Theorem \ref{thm:quotient}, there are $w \in \LF{Y}$ and $a \in \l{\I}{Y}$ such that $R_\I(w) + Q_\I(a) = u$ and $\|w\|+\|a\|_{\l{\I}{Y}} < 1$. Suppose $w = \sum_{(p,q) \in Y^2} b_{(p,q)}(\delta_p-\delta_q)$ with $\sum_{(p,q) \in Y^2}|b_{(p,q)}|d(p,q) + \|a\|_{\l{\I}{Y}} < 1$. Set
$$\tilde{a} := a + \sum_{(p,q) \in Y^2: d(p,q) > \rho(p,q)} b_{(p,q)}(\one_p-\one_q).$$
The following estimate will show that $\tilde{a}$ is a well-defined element of $\l{\I}{Y}$ and moreover that $\sum_{(p,q): d(p,q) \in Y^2 \leq \rho(p,q)}|b_{(p,q)}|d(p,q) + \sum_{p \in Y} |\tilde{a}_{p}|\rad{\I}{p} < 3c^{-1}.$

$$\sum_{(p,q) \in Y^2: d(p,q) \leq \rho(p,q)}|b_{(p,q)}|d(p,q) + \sum_{p \in Y} |\tilde{a}_{p}|\rad{\I}{p}$$
$$\leq \sum_{(p,q) \in Y^2: d(p,q) \leq \rho(p,q)}|b_{(p,q)}|d(p,q)$$
$$+ \|a\|_{\l{\I}{Y}} + \sum_{(p,q) \in Y^2: d(p,q) > \rho(p,q)} |b_{(p,q)}|(\rad{\I}{p}+\rad{\I}{q})$$
$$\overset{\eqref{eq:3c-1}}{\leq} \sum_{(p,q) \in Y^2: d(p,q) \leq \rho(p,q)}|b_{(p,q)}|d(p,q) + \|a\|_{\l{\I}{Y}} + \sum_{(p,q) \in Y^2: d(p,q) > \rho(p,q)} |b_{(p,q)}|3c^{-1}d(p,q)$$
$$\leq 3c^{-1}\left(\sum_{(p,q) \in Y^2}|b_{(p,q)}|d(p,q) + \|a\|_{\l{\I}{Y}}\right) < 3c^{-1}.$$
It is obvious that
$$u = \sum_{(p,q) \in Y^2: d(p,q) \leq \rho(p,q)}b_{(p,q)}(\delta_p-\delta_q) + \sum_{p \in Y} \tilde{a}_{p} \delta_p,$$
proving the lemma.
\end{proof}

\begin{lemma} \label{lem:quotientseparated}
For every ideal $\I \sbs \Cl{X}$ and $u \in \LF{\I}{X}$ with $\|u\| < 1$, there exist $\tilde{w} \in \LF{X}$ and $\tilde{a} \in \l{\I}{X}$ such that $u = R_\I(\tilde{w}) + Q_\I(\tilde{a})$, $\|\tilde{w}\| + \|\tilde{a}\|_{\l{\I}{X}} < 1$, and the following separation property is satisfied: whenever $\tilde{a}_p \cdot \tilde{a}_q < 0$ for some $p \neq q \in X$, then $d(p,q) \geq (\rad{\I}{p}+\rad{\I}{q})/2$.
\end{lemma}

\begin{proof}
Let $\I$ and $u$ be as above. By Theorem \ref{thm:quotient}, there exist $w \in \LF{X}$ and $a \in \l{\I}{X}$ such that $u = R_\I(w) + Q_\I(a)$ and $\|w\|+\|a\|_{\l{\I}{X}} < 1$. Set $\tilde{X} := X \cross X \setminus \Delta$, where $\Delta$ is the diagonal. Consider a family of triples $\{(c_\alpha,p_\alpha,q_\alpha)\}_{\alpha \in A} \sbs \R \cross \tilde{X}$ indexed by a countable ordinal $A$ satisfying
\begin{enumerate}
\item \label{item:11} $c_\alpha \neq 0$ for all $\alpha \in A$,
\item \label{item:12} $d(p_\alpha,q_\alpha) < (\rad{\I}{p_\alpha}+\rad{\I}{q_\alpha})/2$ for all $\alpha \in A$,
\item \label{item:13} $\sum_{\alpha \in A} \|c_\alpha(\delta_{p_\alpha}-\delta_{q_\alpha})\| = \sum_{\alpha \in A} |c_\alpha|d(p_\alpha,q_\alpha) \leq 1$ and \\ $\sum_{\alpha \in A} \left\|c_\alpha(\one_{p_\alpha}-\one_{q_\alpha})\right\|_{\l{\I}{X}} = \sum_{\alpha \in A}, |c_\alpha|(\rad{\I}{p_\alpha}+\rad{\I}{q_\alpha}) \leq 2$
\item \label{item:14} $\|w+x\| + \|a-y\|_{\l{\I}{X}} \leq \|w\| + \|a\|_{\l{\I}{X}} - \sum_{\alpha \in A} |c_\alpha|(\rad{\I}{p_\alpha}+\rad{\I}{q_\alpha}-d(p_\alpha,q_\alpha))$, where $x := \sum_{\alpha \in A}c_\alpha(\delta_{p_\alpha}-\delta_{q_\alpha}) \in \LF{X}$ and $y := \sum_{\alpha \in A} c_\alpha(\one_{p_\alpha}-\one_{q_\alpha}) \in \l{\I}{X}$ ($x$ and $y$ are well-defined by (\ref{item:13})).
\end{enumerate}

The collection of all such indexed families is a poset under the relation $\{(c_\alpha,p_\alpha,q_\alpha)\}_{\alpha \in A} \leq \{(b_\beta,r_\beta,s_\beta)\}_{\beta \in B}$ if $A \sbs B$ and for all $\alpha \in A$, $(c_\alpha,p_\alpha,q_\alpha) = (b_\alpha,r_\alpha,s_\alpha)$. This collection is obviously nonempty since it contains $\emptyset$ (and it is a set since we require the indexing sets to be countable ordinals). Every chain is bounded above by the union of the families in the chain. This union belongs to our poset since the four items above are all preserved under limits and items (\ref{item:11}) and (\ref{item:13}) and the fact that $p_\alpha \neq q_\alpha$ for all $\alpha \in A$ imply the indexing set must be countable. Thus we get a maximal element $\{(c_\alpha,p_\alpha,q_\alpha)\}_{\alpha \in A}$ by Zorn's lemma. Set $\tilde{x} := \sum_{\alpha \in A}c_\alpha(\delta_{p_\alpha}-\delta_{q_\alpha}) \in \LF{X}$ and $\tilde{y} := \sum_{\alpha \in A} c_\alpha(\one_{p_\alpha}-\one_{q_\alpha}) \in \l{\I}{X}$. Set $\tilde{w} := w + \tilde{x}$ and $\tilde{a} := a - \tilde{y}$. We complete the proof by showing that $\tilde{w},\tilde{a}$ satisfy the three desired properties. Clearly $R_\I(\tilde{w}) + Q_\I(\tilde{a}) = R_\I(w) + Q_\I(a) = u$, and (\ref{item:12}) and (\ref{item:14}) imply $\|\tilde{w}\|+\|\tilde{a}\|_{\l{\I}{X}} \leq \|w\|+\|a\|_{\l{\I}{X}}$. Since $\|w\|+\|a\|_{\l{\I}{X}} < 1$, it only remains to show the separation property.

Suppose that it does not hold. Let $p \neq q \in X$ such that $\tilde{a}_p \cdot \tilde{a}_q < 0$ and $d(p,q) < (\rad{\I}{p}+\rad{\I}{q})/2$. Without loss of generality, we may assume $\tilde{a}_p > 0$ and $\tilde{a}_q < 0$. Then there are two cases: $|\tilde{a}_p| \geq |\tilde{a}_q|$ and $|\tilde{a}_p| \leq |\tilde{a}_q|$. We will assume the first case holds - the second case can be treated similarly. We consider the new family of triples $\{(c_\alpha,p_\alpha,q_\alpha)\}_{\alpha \in A} \cup \{(-\tilde{a}_q,p,q)\}$ indexed by the successor of $A$. If we can show this family belongs to our collection, this will contradict maximality. Items (\ref{item:11}) and (\ref{item:12}) clearly hold for this family. Next we verify (\ref{item:14}). Letting $x,y$ be as defined in (\ref{item:14}) for the family under consideration, we have

$$\|w+x\| + \|a-y\|_{\l{\I}{X}} = \|w+\tilde{x}-\tilde{a}_q(\delta_p-\delta_q)\| + \|a-\tilde{y}+\tilde{a}_q(\one_p-\one_q)\|_{\l{\I}{X}}$$
$$= \|\tilde{w}-\tilde{a}_q(\delta_p-\delta_q)\| + \|\tilde{a}+\tilde{a}_q(\one_p-\one_q)\|_{\l{\I}{X}}$$
$$= \|\tilde{w}-\tilde{a}_q(\delta_p-\delta_q)\| + \left\|(\tilde{a}_q+\tilde{a}_p)\one_p+\sum_{s \in X \setminus \{p,q\}}\tilde{a}_s\one_s\right\|_{\l{\I}{X}}$$
$$\leq \|\tilde{w}\|+|\tilde{a}_q|d(p,q) + \left\|(\tilde{a}_q+\tilde{a}_p)\one_p+\sum_{s \in X \setminus \{p,q\}}\tilde{a}_s\one_s\right\|_{\l{\I}{X}}$$
$$= \|\tilde{w}\|+|\tilde{a}_q|d(p,q) + |\tilde{a}_q+\tilde{a}_p|\rad{\I}{p}+\sum_{s \in X \setminus \{p,q\}}|\tilde{a}_s|\rad{\I}{s}$$
$$= \|\tilde{w}\|+|\tilde{a}_q|(\rad{\I}{p}+\rad{\I}{q}) - |\tilde{a}_q|(\rad{\I}{p}+\rad{\I}{q}-d(p,q))$$
$$+ |\tilde{a}_q+\tilde{a}_p|\rad{\I}{p}+\sum_{s \in X \setminus \{p,q\}}|\tilde{a}_s|\rad{\I}{s}$$
$$= \|\tilde{w}\|-\tilde{a}_q(\rad{\I}{p}+\rad{\I}{q}) - |\tilde{a}_q|(\rad{\I}{p}+\rad{\I}{q}-d(p,q))$$
$$+ (\tilde{a}_q+\tilde{a}_p)\rad{\I}{p}+\sum_{s \in X \setminus \{p,q\}}|\tilde{a}_s|\rad{\I}{s}$$
$$= \|\tilde{w}\|+\sum_{s \in X}|\tilde{a}_s|\rad{\I}{s} - |\tilde{a}_q|(\rad{\I}{p}+\rad{\I}{q}-d(p,q))$$
$$= \|\tilde{w}\|+\|\tilde{a}\|_{\l{\I}{X}} - |\tilde{a}_q|(\rad{\I}{p}+\rad{\I}{q}-d(p,q))$$
$$\leq \|w\| + \|a\|_{\l{\I}{X}}$$
$$ - \sum_{\alpha \in A} |c_\alpha|(\rad{\I}{p_\alpha}+\rad{\I}{q_\alpha}-d(p_\alpha,q_\alpha)) - |\tilde{a}_q|(\rad{\I}{p}+\rad{\I}{q}-d(p,q))$$
where the last inequality follows from (\ref{item:14}) applied to the family $\{(c_\alpha,p_\alpha,q_\alpha)\}_{\alpha \in A}$ (the maximal family given to us by Zorn's lemma). This shows (\ref{item:14}).

Finally, we use (\ref{item:14}) to verify (\ref{item:13}). We have
$$0 \leq \|w\| + \|a\|_{\l{\I}{X}}$$
$$- \sum_{\alpha \in A} |c_\alpha|(\rad{\I}{p_\alpha}+\rad{\I}{q_\alpha}-d(p_\alpha,q_\alpha)) - |\tilde{a}_q|(\rad{\I}{p}+\rad{\I}{q}-d(p,q))$$
$$\leq 1 - \sum_{\alpha \in A} |c_\alpha|(\rad{\I}{p_\alpha}+\rad{\I}{q_\alpha}-d(p_\alpha,q_\alpha)) - |\tilde{a}_q|(\rad{\I}{p}+\rad{\I}{q}-d(p,q))$$
$$\leq 1 - \sum_{\alpha \in A} |c_\alpha|(\rad{\I}{p_\alpha}+\rad{\I}{q_\alpha})/2 - |\tilde{a}_q|(\rad{\I}{p}+\rad{\I}{q})/2,$$
which gives us
$$\sum_{\alpha \in A} |c_\alpha|d(p_\alpha,q_\alpha) + |\tilde{a}_q|d(p,q)$$
$$\leq \sum_{\alpha \in A} |c_\alpha|(\rad{\I}{p_\alpha}+\rad{\I}{q_\alpha})/2 + |\tilde{a}_q|(\rad{\I}{p}+\rad{\I}{q})/2 \leq 1.$$
\end{proof}

\subsection{Countable $X$, General $\I$}
\begin{lemma} \label{lem:quotientfinite}
Assume $X$ is countable with enumeration $X = \{p_i\}_{i=1}^\infty$. Fix an ideal $\I \sbs \Cl{X}$ and set $\J := \ulm{\I}$. For every $u \in \LF{\I}{X}$ with $\|u\| < 1$ and collection of radii $\{r_p\}_{p \in X}$ with $r_p \in (0,\rad{\J}{p})$, there exist $w \in \LF{X}$ and $a \in \l{\I}{X}$ such that $u = R_\I(w) + Q_\I(a)$, $\|w\| + \|a\|_{\l{\I}{X}} < 1$, and $\supp{a} \cap B_{r_p}(p)$ is finite for each $p \in X$.
\end{lemma}

\begin{proof}
Let $u$ and $r_p$ be as above. By Lemma \ref{lem:quotientseparated}, there exist $w' \in \LF{X}$ and $a' \in \l{\I}{X}$ such that $u = R_\I(w') + Q_\I(a')$, $\|w'\|+\|a'\|_{\l{\I}{X}} < 1$, and whenever $a'_p \cdot a'_q < 0$ for some $p \neq q \in X$, then $d(p,q) \geq (\rad{\I}{p}+\rad{\I}{q})/2$.

Let $i \geq 1$. Since $r_{p_i} < \rad{\J}{p_i}$, Proposition \ref{prop:radIinf} implies \\ $\inf_{q \in  B_{r_{p_i}}(p_i)} \rad{\I}{q} \geq \theta$ for some $\theta > 0$. Set $S^+ := \{p \in X: a'_p > 0\}$ and $S^- := \{p \in X: a'_p < 0\}$. Let $F \sbs (S^+ \cup S^-) \cap B_{r_{p_i}}(p_i)$ be an arbitrary finite subset. Consider the functions $f^+,f^-: X \to \R$ defined by $f^+(x) = \max(0,1-2\theta^{-1}d(F \cap S^+,x))$ and $f^-(x) = \max(0,1-2\theta^{-1}d(F \cap S^-,x))$. We make several observations.
\begin{itemize}
\item $\supp{f^+} \cup \supp{f^-} \sbs \cup_{q \in F} B_{\theta/2}(q)$. This is a finite union of sets in $\I$ and is thus itself in $\I$.
\item $\|f^+\|,\|f^-\| \leq 2\theta^{-1}$.
\item $\im(f^+),\im(f^-) \sbs [0,1]$.
\item $f^+(q) = 1$ for all $q \in F \cap S^+$ and $f^-(q) = 1$ for all $q \in F \cap S^-$.
\item Since $d(p,q) \geq (\rad{\I}{p}+\rad{\I}{q})/2$ whenever $a'_p \cdot a'_q < 0$, $\supp{f^+} \cap S^- = \supp{f^-} \cap S^+ = \emptyset$.
\end{itemize}
Together these imply
$$4\theta^{-1} \geq \|f^+-f^-\|\|Q_\I(a')\|_\I \geq |Q_\I(a')(f^+-f^-)| = \left|\sum_{q \in \supp{a'}} a'_q(f^+(q)-f^-(q))\right|$$
$$= \left|\sum_{q \in S^+} a'_qf^+(q) + \sum_{q \in S^-} (-a'_q)f^-(q)\right| = \sum_{q \in S^+} |a'_q|f^+(q) + \sum_{q \in S^-} |a'_q|f^-(q)$$
$$\geq \sum_{q \in F \cap S^+} |a'_q|f^+(q) + \sum_{q \in F \cap S^-} |a'_q|f^-(q) = \sum_{q \in F} |a'_q|.$$
Since $F \sbs (S^+ \cup S^-) \cap B_{r_{p_i}}(p_i)$ was arbitrary, we get $\sum_{q \in B_{r_{p_i}}(p_i)} |a'_q| \leq 4\theta^{-1}$. Thus, there is some cofinite subset $CF_i \sbs B_{r_{p_i}}(p_i)$ such that $\sum_{q \in CF_i} |a'_q| < 2^{-i}\eps$, where $\eps$ is any number in $(0,1-\|w'\|-\|a'\|_{\l{\I}{X}})$. Then setting $CF_\infty := \cup_{i=1}^\infty CF_i$, we get $\sum_{q \in CF_\infty} |a'_q| \leq \eps$. Then we set $w := w' + \sum_{q \in CF_\infty} a'_q\delta_q$ and $a := \sum_{q \in X \setminus CF_\infty} a'_q\one_q \in \l{\I}{X}$. Then $R_\I(w) + Q_\I(a) = R_I(w') + Q_\I(a') = u$, $\|w\| + \|a\|_{\l{\I}{X}} \leq \|w'\| + \eps + \|a'\|_{\l{\I}{X}} < 1$, and $\supp{a} \cap B_{r_{p_i}}(p_i) \sbs B_{r_{p_i}}(p_i) \setminus CF_i$ for each $i \geq 1$.
\end{proof}

\subsection{Countable and Discrete $X$, $\I = \F_n$}

\begin{definition}
Recall that $\F \sbs \Cl{X}$ is the ideal of finite subsets of $X$. Suppose $X$ is discrete. Then for each $p \in X$, $\one_p \in \Lip{\F}{X}$. For any $\lambda \in \Lip{\F}{X}^*$, we write $\lambda = \sum_{p \in X} c_p\delta_p$ if $\lambda(\one_p) = c_p$ for each $p \in X$. This gives a well-defined linear injection $\Lip{\F}{X}^* \to \R^{X}$. We emphasize that the notation is purely formal and that we make no attempt to interpret it as a limit of finite sums. Whenever $\I \sbs \Cl{X}$ is an ideal such that $\F \sbs \I$ and $R^\I_{\F}: \LF{\I}{X} \to \LF{\F}{X}$ is injective and $v \in \LF{\I}{X}$, we also write $v = \sum_{p \in X} c_p\delta_p$ to mean $R^\I_{\F}(v) = \sum_{p \in X} c_p\delta_p$. Since $R^\I_{\F}$ is injective, this again gives a well-defined linear injection $\LF{\I}{X} \to \R^X$. For example, Theorem \ref{thm:injective} and induction imply $R^{\F_n}_{\F}$ is injective for all $n \geq 0$. Whenever $v = \sum_{p \in X} c_p\delta_p$, we define $\boldsymbol{\supp{v}} := \{p \in X: c_p \neq 0\} = \{p \in X: v(\one_p) \neq 0\}$.
\end{definition}

A natural question is whether $v \in \LF{\I}{\supp{v}}$. While we believe this is true, we prove only a very restricted version in Lemma \ref{lem:supp}. This lemma will be coupled with Lemma \ref{lem:quotientclose} in the proof of Theorem \ref{thm:semiembedding}.

\begin{lemma} \label{lem:supp}
Assume $X$ is discrete. Fix an ideal $\I \sbs \Cl{X}$ such that $\F \sbs \I$ and $R^\I_{\F}$ is injective. For all $Y,F \sbs X$ with $F$ finite, $v \in \LF{\I}{X}$, and $\eps>0$, if $\supp{v} \sbs Y \cup F$, then $v \in \LF{\I}{B_{\eps}(Y) \cup F}$.
\end{lemma}

\begin{proof}
Let $Y,F,v,\eps$ be as above. Assume $\supp{v} \sbs Y \cup F$. Let $\psi': X \to [0,1]$ be a Lipschitz function equal to 1 on $Y$ and 0 on $X \setminus B_{\eps}(Y)$. Since $F$ is finite and $X$ is discrete, $\one_F: X \to [0,1]$ is Lipschitz. Set $\psi := \max(\psi'+\one_F,1): X \to [0,1]$, so that $\psi$ is Lipschitz, equals 1 on $Y \cup F$, and equals 0 on $X \setminus (B_{\eps}(Y) \cup F)$. The map $f \mapsto \psi \cdot f$ is a bounded linear operator $\Lip{\I}{X} \to \Lip{\I}{X}$, and its dual $T: \Lip{\I}{X}^* \to \Lip{\I}{X}^*$ maps $\LFf{\I}{X}$ into $\LFf{\I}{B_{\eps}(Y) \cup F}$ and thus also $\LF{\I}{X}$ into $\LF{\I}{B_{\eps}(Y) \cup F}$. We will show $v = T(v)$, completing the proof because $T(v) \in \LF{\I}{B_{\eps}(Y) \cup F}$.

By hypothesis, $R^\I_{\F}$ is injective, so it suffices to show $R^\I_{\F}(v) = R^\I_{\F}(T(v))$. Since $\{\one_p\}_{p \in X}$ spans $\Lip{\F}{X}$, it suffices to show \\ $[R^\I_{\F}(v)](\one_p) = [R^\I_{\F}(T(v))](\one_p)$ for every $p \in X$. Let $p \in X$. We have
$$[R^\I_{\F}(v)](\one_p) = v(\one_p)$$
and
$$[R^\I_{\F}(T(v))](\one_p) = [T(v)](\one_p) = v(\psi \cdot \one_p) = \psi(p)v(\one_p)$$
so we need to show $v(\one_p) = \psi(p)v(\one_p)$. There are two cases to consider, $p \in \supp{v}$ and $p \in X \setminus \supp{v}$. In the first case, $\psi(p) = 1$ by definition of $\psi$ and the assumption $\supp{v} \sbs Y \cup F$, and in the second case, $v(\one_p) = 0$ by definition of $\supp{v}$. In either case we have $v(\one_p) = \psi(p)v(\one_p)$.
\end{proof}

\begin{lemma} \label{lem:sequence1}
Assume $X$ is discrete. Fix an ideal $\I \sbs \Cl{X}$ such that $\F \sbs \I$ and $R^\I_{\F}$ is injective. Set $\J := \ulm{\I}$ so that Theorem \ref{thm:injective} implies $R^\J_{\F}$ is also injective (for example, $\I = \F_n$ and $\J = \F_{n+1}$). For any $p \in X$, $r \in (0,\rad{\J}{p})$, $u \in \LF{\I}{X}$ such that $\supp{u} \cap B_{r}(p)$ is finite, and sequence $v_n \in \LF{\J}{X}$ such that $\sup_n \|v_n\|_\J < 1$ and $R^\J_\I(v_n) \to u$, and $c \in (0,1)$, there exists another sequence $\tilde{v}_n$ such that
\begin{itemize}
\item $\sup_n \|\tilde{v}_n\|_\J < 1$,
\item $R^\J_\I(\tilde{v}_n) \to u$,
\item $\supp{\tilde{v}_n} \sbs \supp{v_n}$,
\item and $\supp{\tilde{v}_n} \cap B_{cr}(p) = \supp{u} \cap B_{cr}(p)$.
\end{itemize}
\end{lemma}

\begin{proof}
Let $p,r,u,c$ be as above. Set $A_1 := B_{r}(p)$ and $A_2 := B_{cr}(p)$ so that $A_2 \sbs A_1 \in \J$. Let $\psi$ be a Lipschitz bump function taking values in $[0,1]$, identically 1 on $A_2$, and 0 on $X \setminus A_1$ (for example, $\psi(q) = \max\{0,\min\{1,(1-c)^{-1}-(1-c)^{-1}r^{-1}d(p,q)\}\}$). The function $\psi$ induces a bounded linear operator $\Lip{A_1} \to \Lip{\langle A_1 \rangle}{X}$ given by $f \mapsto \psi\tilde{f}$, where $\tilde{f}$ is any norm preserving extension of $f$ to all of $X$ (which exists by McShane). It's clear that this map does not depend on the choice of $\tilde{f}$, is linear, and has operator norm depending only on $c$ and $r$. The dual $T$ of this bounded operator maps $\LFf{\langle A_1 \rangle}{X}$ into $\LFf{A_1}$, and thus also $\LF{\langle A_1 \rangle}{X} into \LF{A_1} \sbs \LF{X}$.

Since $A_1 \in \J = \ulm{\I}$, Lemma \ref{lem:isomorphism} implies $R^{\langle A_1,\I \rangle}_\I: \LF{\langle A_1,\I \rangle}{X} \to \LF{\I}{X}$ is an isomorphism. Thus, $S := R_\J \comp T \comp R^{\langle A_1,\I \rangle}_{\langle A_1 \rangle} \comp (R^{\langle A_1,\I \rangle}_{\I})^{-1}: \LF{\I}{X} \to \LF{\J,X}$ is a bounded linear map given by
$$S\left(\sum_{q \in X} c_q\delta_q\right) = \sum_{q \in X} c_q\psi(q)\delta_q = \sum_{q \in A_2} c_q\delta_q + \sum_{q \in A_1 \setminus A_2} c_q\psi(q)\delta_q.$$
Suppose $v_n = \sum_{p \in X} c_{p,n}\delta_p$ and $u = \sum_{p \in X} c_p\delta_p$. Then since $R^\J_\I(v_n) \to u$, $S(R^\J_\I(v_n)) \to S(u) \in \LF{\J}{X}$, which translates to
$$\sum_{q \in A_2} c_{q,n}\delta_{q} + \sum_{q \in A_1 \setminus A_2} c_{q,n}\psi(q)\delta_q \to \sum_{q \in A_2} c_{q}\delta_{q} + \sum_{q \in A_1 \setminus A_2} c_{q}\psi(q)\delta_q \in \LF{\J}{X}.$$

Set $F := \supp{u} \cap A_1$, which is finite by assumption. Since $X$ is discrete, $\one_q$ is a Lipschitz function for all $q \in X$ and hence $\sum_{q \in F} c_{q,n}\psi(q)\delta_q \to \sum_{q \in F} c_{q}\psi(q)\delta_q \in \LF{\J}{X}$. Combining these implies
$$\sum_{q \in A_2 \setminus F} c_{q,n}\delta_{q} + \sum_{q \in A_1 \setminus (A_2 \cup F)} c_{q,n}\psi(q)\delta_q \to 0 \in \LF{\J}{X}.$$
Choose $N$ large enough so that
$$\left\| \sum_{q \in A_2 \setminus F} c_{q,n}\delta_{q} + \sum_{q \in A_1 \setminus (A_2 \cup F)} c_{q,n}\psi(q)\delta_q \right\|_\J < \frac{1 - \sup_n \|v_n\|}{2}$$
for all $n \geq N$. Then consider the $\LF{\J}{X}$-valued sequence $\nu_n$ defined by
$$\nu_n = \left\{\begin{matrix} 0 & n < N \\ \sum_{q \in A_2 \setminus F} c_{q,n}\delta_{q} + \sum_{q \in A_1 \setminus (A_2 \cup F)} c_{q,n}\psi(q)\delta_q & n \geq N \end{matrix}\right..$$
We have $\sup_n \|\nu_n\|_\J < 1 - \sup_n \|v_n\|$ and $\|\nu_n\|_\J \to 0$. Set $\tilde{v}_n := v_n - \nu_n$. Thus, $\sup_n \|\tilde{v}_n\|_\J < 1$, $R^\J_\I(\tilde{v}_n) \to u$, $\supp{\tilde{v}_n} \sbs \supp{v_n}$, and $\supp{\tilde{v}_n} \cap A_2 = F \cap A_2 = \supp{u} \cap A_2$ for all $n \geq N$. By omitting all the terms before $N$ and shifting the index, we may assume $\supp{\tilde{v}_n} \cap A_2 = \supp{u} \cap A_2$ for all $n \geq 1$.
\end{proof}

\begin{lemma} \label{lem:sequence2}
Assume $X$ is countable with enumeration $X = \{p_i\}_{i=1}^\infty$. Assume $X$ is discrete. Fix an ideal $\I \sbs \Cl{X}$ such that $\F \sbs \I$ and $R^\I_{\F}$ is injective. Set $\J := \ulm{\I}$ so that Theorem \ref{thm:injective} implies $R^\J_{\F}$ is also injective (for example, $\I = \F_n$ and $\J = \F_{n+1}$). For any collection of radii $\{r_p\}_{p \in X}$ with $r_p \in (0,\rad{\J}{p})$, $u \in \LF{\I}{X}$ such that $\supp{u} \cap B_{r_p}(p)$ is finite for all $p \in X$, sequence $v_n \in \LF{\J}{X}$ such that $\sup_n \|v_n\|_\J < 1$ and $R^\J_\I(v_n) \to u$, and $c \in (0,1)$, there exists another sequence $\tilde{v}_n \in \LF{\J}{X}$ such that
\begin{itemize}
\item $\|\tilde{v}_n\|_\J < 1$ for every $n \geq 1$,
\item $R^\J_\I(\tilde{v}_n) \to u$,
\item for all $p_i \in X$ and $n \geq i$, $\supp{\tilde{v}_n} \cap B_{cr_{p_i}}(p_i) \sbs \supp{u} \cap B_{cr_{p_i}}(p_i)$.
\end{itemize}
\end{lemma}

\begin{proof}
Let $\{r_p\}_{p \in X}$, $u$, $v_n$, and $c$ be as above. Applying Lemma \ref{lem:sequence1} with $p = p_1$, $r = r_{p_1}$, $u = u$, $v_n = v_n$, and $c = c$, we get a sequence $\tilde{v}_n^1$ with $\sup_n \|\tilde{v}_n^1\|_\J < 1$, $R^\J_\I(\tilde{v}_n^1) \to u$, and $\supp{\tilde{v}_n^1} \cap B_{cr_{p_{1}}}(p_{1}) = \supp{u} \cap B_{cr_{p_{1}}}(p_{1})$. Then we apply Lemma \ref{lem:sequence1} again with $p = p_2$, $r = r_{p_2}$, $u = u$, $v_n = \tilde{v}_n^1$, and $c=c$ to get a sequence $\tilde{v}_n^2$ with $\sup_n \|\tilde{v}_n^2\|_\J < 1$, $R^\J_\I(\tilde{v}_n^2) \to u$, and $\supp{\tilde{v}_n^2} \cap B_{cr_{p_{2}}}(p_{2}) = \supp{u} \cap B_{cr_{p_{2}}}(p_{2})$ and also $\supp{\tilde{v}_n^2} \cap B_{cr_{p_{1}}}(p_{1}) \sbs \supp{\tilde{v}_n^1} \cap B_{cr_{p_{1}}}(p_{1}) = \supp{u} \cap B_{cr_{p_{1}}}(p_{1})$. Iterating, we get, for all $i$, sequences $\tilde{v}_n^i$ such that $\sup_n \|\tilde{v}_n^i\|_\J < 1$, $R^\J_\I(\tilde{v}_n^i) \to u$, and $\supp{\tilde{v}_n^i} \cap B_{cr_{p_{i'}}}(p_{i'}) \sbs \supp{u} \cap B_{cr_{p_{i'}}}(p_{i'})$ for all $i' \leq i$. By diagonalizing, we can find a single sequence $\tilde{v}_n$ with the required properties. 
\end{proof}

\begin{theorem} \label{thm:semiembedding}
Assume $X$ is countable with enumeration $X = \{p_i\}_{i=1}^\infty$. Assume $X$ is discrete. Fix an ideal $\I \sbs \Cl{X}$ such that $\F \sbs \I$ and $R^\I_{\F}$ is injective. Set $\J := \ulm{\I}$ so that Theorem \ref{thm:injective} implies $R^\J_{\F}$ is also injective (for example, $\I = \F_n$ and $\J = \F_{n+1}$). The map $R^{\J}_\I: \LF{\J,X} \to \LF{\I}{X}$ is a 7-semi-embedding.
\end{theorem}

\begin{proof}
By Theorem \ref{thm:injective}, $R^\J_\I$ is injective, so it remains to show \\ $\closure{R^\J_\I(B_{\LF{\J}{X}})} \sbs R^\J_\I(7B_{\LF{\J}{X}})$. Suppose $v'_n \in\LF{\J,X}$ and $R^{\J}_\I(v'_n) \to u' \in \LF{\I}{X}$ with $\sup_n \|v'_n\|_\J < 1$ and $\|u'\|_\I < 1$. Let $c \in (0,1)$. Set $r_p := c \cdot \rad{\J}{p}$ for each $p \in X$. By Lemma \ref{lem:quotientfinite}, $u' = R_\I(w) + u$ for some $w \in \LF{X}$ and $u \in \LF{\I}{X}$ with $\|w\| + \|u\|_\I < 1$ and $\supp{u} \cap B_{r_p}(p)$ is finite for each $p \in X$. Set $v_n := v'_n - R_{\J}(w)$ so that $\sup_n\|v_n\|_\J < 2$ and $R^{\J}_\I(v_n) \to u$. Then we apply Lemma \ref{lem:sequence2} with $r_p = r_p$, $u = u$, $v_n = v_n$, $c = c$ to obtain a sequence $\tilde{v}_n$ such that
\begin{enumerate}
\item $\|\tilde{v}_n\|_\J < 2$ for all $n \geq 1$,
\item $R^\J_\I(\tilde{v}_n) \to u$,
\item \label{item:3} for all $p_i \in X$ and $n \geq i$, $\supp{\tilde{v}_n} \cap B_{cr_{p_i}}(p_i) \sbs \supp{u} \cap B_{cr_{p_i}}(p_i)$.
\end{enumerate}
Fix $n \geq 1$, and set $Y := X \setminus (\cup_{i \leq n} B_{cr_{p_i}}(p_i))$ and $F := \cup_{i \leq n} (\supp{u} \cap B_{cr_{p_i}}(p_i))$ so that $F$ is finite and $\supp{\tilde{v}_n} \sbs Y \cup F$. Choose $\eps>0$ small enough so that $B_{\eps}(Y) \sbs X \setminus (\cup_{i \leq n} B_{c^2r_{p_i}}(p_i))$. Then Lemma \ref{lem:supp} implies
\begin{equation} \label{eq:supp}
\tilde{v}_n \in \LF{\J}{Z_n}
\end{equation}
where $Z_n := F \cup (X \setminus (\cup_{i \leq n} B_{c^2r_{p_i}}(p_i)))$.
Set $\rho(p,q) := c^3\min(\rad{\J}{p},\rad{\J}{q}) \leq c^2r_{p}$ and for each $j \geq 1$, $Z_{n,j} := B_{\rho(p_j,q)}(p_j) \cap Z_n$. When $n \geq j$,
\begin{equation} \label{eq:Znj}
Z_{n,j} = B_{\rho(p_j,q)}(p_j) \cap F \sbs \supp{u} \cap B_{cr_{p_j}}(p_j).
\end{equation}

Suppose $\tilde{v}_n = \sum_{p \in X} c_{p,n}\delta_p$ and $u = \sum_{p \in X} c_p\delta_p$, so that $c_{p,n} \to c_p$ for each $p$. Since $\|\tilde{v}_n\|_\J < 2$, \eqref{eq:supp} and Lemma \ref{lem:quotientclose} allow us write
\begin{equation} \label{eq:c=b}
\sum_{p \in X} c_{p,n}\delta_p = \sum_{\substack{(p,q) \in Z_n^2 \\ d(p,q) \leq \rho(p,q)}}b_{(p,q),n}(\delta_p-\delta_q) + \sum_{p \in Z_n} a_{p,n} \delta_p
\end{equation}
with
$$\sum_{\substack{(p,q) \in Z_n^2 \\ d(p,q) \leq \rho(p,q)}}|b_{(p,q),n}|d(p,q) + \sum_{p \in Z_n} |a_{p,n}|\rad{\J}{p} < 6c^{-3}.$$
For the sake of notational convenience, we'll define $b_{(p,q),n} := 0$ if $(p,q) \in X^2 \setminus Z_n^2$ or $d(p,q) > \rho(p,q)$ and $a_{p,n} := 0$ if $p \in X \setminus Z_n$ so that the above equations remain true when the first terms are summed over $X^2$ and the second over $X$.

Fix a nonprincipal ultrafilter $\U$ on $\N$. It is easy to see that $|b_{(p,q),n}| \leq 6c^{-3}d(p,q)^{-1}$ and $|a_{p,n}| \leq 6c^{-3}\rad{\J}{p}^{-1}$ for all $p\neq q \in X$ and $n \geq 1$, and thus $b_{(p,q)} := \ulim_{n \to \infty} b_{(p,q),n}$ and $a_p := \ulim_{n \to \infty} a_{p,n}$ exist.
For each $p \in X$, evaluating each side of \eqref{eq:c=b} at the Lipschitz function $\one_p$ yields
$$c_{p,n} = \sum_{q \in X} (b_{(p,q),n}-b_{(q,p),n}) + a_{p,n}$$
$$=\sum_{\substack{q \in Z_n \\ d(p,q) \leq \rho(p,q)}} (b_{(p,q),n}-b_{(q,p),n}) + a_{p,n}.$$
Then for each $j \geq 1$, combining this with the definition of $Z_{n,j}$ gives
$$c_{p_j,n} = \sum_{q \in Z_{n,j}} (b_{(p_j,q),n}-b_{(q,p_j),n}) + a_{p_j,n}.$$
By \eqref{eq:Znj}, the set over which this sum is performed, $Z_{n,j}$, is eventually (depending on $j$) contained in the fixed finite set $\supp{u} \cap B_{cr_{p_j}}(p_j)$. Thus we can interchange the sum and $\U$-limit to obtain
$$c_{p_j} = \ulim_{n \to \infty} c_{p_j},n = \ulim_{n \to \infty} \sum_{q \in X} (b_{(p_j,q),n}-b_{(q,p_j),n}) + a_{p_j,n}$$
$$= \sum_{q \in X} \ulim_{n \to \infty} (b_{(p_j,q),n}-b_{(q,p_j),n}) + \ulim_{n \to \infty} a_{p_j,n}$$
$$= \sum_{q \in X} (b_{(p_j,q)}-b_{(q,p_j)}) + a_{p_j}$$
and thus
$$u = \sum_{p_j \in X} c_{p_j}\delta_{p_j} = \sum_{p_j \in X} \sum_{q \in X}(b_{(p_j,q)}-b_{(q,p_j)})\delta_{p_j} + a_{p_j}\delta_{p_j}$$
$$= \sum_{(p,q) \in X^2}b_{(p,q)}(\delta_p-\delta_q) + \sum_{p \in X} a_p \delta_p$$
and
$$\sum_{(p,q) \in X^2} \|b_{(p,q)}(\delta_p-\delta_q)\|_\J + \sum_{p \in X} \|a_{p} \delta_p\|_\J$$
$$\leq \sum_{(p,q) \in X^2}|b_{(p,q)}|d(p,q) + \sum_{p \in X} |a_{p}|\rad{\J}{p}$$
$$= \sum_{(p,q) \in X^2}\ulim_{n \to \infty}|b_{(p,q),n}|d(p,q) + \sum_{p \in X} \ulim_{n \to \infty}|a_{p,n}|\rad{\J}{p}$$
$$\leq \ulim_{n \to \infty} \sum_{(p,q) \in X^2}|b_{(p,q),n}|d(p,q) + \sum_{p \in X}|a_{p,n}|\rad{\J}{p} \leq 6c^{-3}$$
where the second to last inequality is Fatou's lemma. Thus, we set $v = \sum_{(p,q) \in X^2}b_{(p,q)}(\delta_p-\delta_q) + \sum_{p \in X} a_{p} \delta_p \in \LF{\J}{X}$, so the last two formulae give us $R^\J_\I(v) = u$ and $\|v\|_\J \leq 6c^{-3}$. Then $R^\J_\I(R_\J(w)+v) = R_\I(w) + u = u'$ and $\|R_\J(w)+v\|_\J \leq 1 + 6c^{-3}$. Taking $c \nearrow 1$ shows $\closure{R^\J_\I(B_{\LF{\J}{X}})} \sbs R^\J_\I(7B_{\LF{\J}{X}})$.
\end{proof}

\subsection{Predual of $\LF{\F}{X}$} \label{ss:predual} ~\\
\indent In this subsection, we observe that $\LF{\F}{X}$ is a dual space. This is needed for the base case in the proof of Theorem \ref{thm:RNP}.

\begin{definition}
Recall that $\lip{X}$ is the closed subspace of $\Lip{X}$ defined by $f \in \lip{X}$ if, for every $\eps>0$, there exists $\delta > 0$ such that $0<d(x,y)<\delta$ implies $|f(y)-f(x)| \leq \eps d(x,y)$. Whenever $\I \sbs \Cl{X}$ is an ideal, we define $\boldsymbol{\lip{\I}{X}} := \lip{X} \cap \Lip{\I}{X}$. We say that $\lip{\I}{X}$ \emph{separates points uniformly} if there exists $a < \infty$ such that for every $F \sbs X$ finite and $g \in B_{\Lip{\I}{X}}$, there exists $f \in \lip{\I}{X}$ such that $\|f\| \leq a$ and $f$ and $g$ agree on $F$.
\end{definition}

We recall the following theorem from \cite{W}. The theorem is stated there for compact metric spaces, but the proof easily extends to compactly supported Lipschitz functions on a locally compact metric space.

\begin{theorem}[\cite{W}, Theorem 3.3.3] \label{thm:predual}
Let $\K \sbs \Cl{X}$ denote the ideal of compact subsets. If $\lip{\K}{X}$ separates points uniformly, then $\LF{\K}{X} = \lip{\K}{X}^*$.
\end{theorem}

\begin{corollary} \label{cor:predual}
If $X$ is discrete, then $\LF{\F}{X} = \Lip{\F}{X}^*$.
\end{corollary}

\begin{proof}
Assume $X$ is discrete. Then $\K = \F$ and \\ $\lip{\K}{X} = \lip{\F}{X} = \Lip{\F}{X}$. The corollary then follows from Theorem \ref{thm:predual}.
\end{proof}

\section{Open Questions}
We conclude by listing additional open questions learned from Anton\'in Proch\'azka during a Banach space theory webinar.

\begin{question}
Are the Radon-Nikod\'ym and Schur properties equivalent for Lipschitz free spaces?
\end{question}

\begin{question}
Does a Lipschitz free space failing to have RNP necessarily contain an isomorph of $L^1(\R)$?
\end{question}

\begin{question}
Does $\LF{X}$ failing to have the RNP imply $X$ contains a biLipschitz copy of a compact, positive measure subset of $\R$?
\end{question}

\begin{question}
Does $\LF{X}$ containing an isomorph of $L^1(\R)$ imply $X$ contains a biLipschitz copy of a compact, positive measure subset of $\R$?
\end{question}

The converse to each of the last three questions has a positive answer. This is because $L^1(\R)$ fails to have the RNP and $\LF{X}$ is isomorphic to $L^1(\R)$ whenever $X$ is biLipschitz equivalent to a compact, positive measure subset of $\R$ (\cite[Corollary 3.4]{Godard}).

\bibliographystyle{abbrv}
\bibliography{lipschitzfreebib}

\end{document}